\documentclass[a4paper,10pt]{article}
\usepackage[centertags]{amsmath}
\usepackage{amsfonts}
\usepackage{amssymb}
\usepackage{amsthm}
\usepackage{epsfig}
\usepackage{setspace}
\usepackage{ae}
\usepackage{eucal}
\usepackage[usenames]{color}%
\usepackage{graphicx}

\theoremstyle{plain}
\newtheorem{thm}{Theorem}[section]
\newtheorem{lem}[thm]{Lemma}
\newtheorem{cor}[thm]{Corollary}

\theoremstyle{definition}

\theoremstyle{remark}

\newtheorem{rem}[thm]{Remark}

\newcommand{\E}{\mathbb{E}}

\title{Comparison results   for highly degenerate parabolic equations with univariate convex data and optimal strategies  for options on trading accounts }

\author{J\"org Kampen, Jan Vecer}
\begin{document}

\maketitle

\begin{abstract}
For linear multivariate purely second order highly degenerated  parabolic equations  with univariate convex data, 
monotonicity of  coefficient matrices implies monotonicity of the related value functions.  For multivariate data, comparison 
holds only for trivial coefficients, where we recover and extend existing results
for uniformly elliptic equations to highly degenerate equations by a different method of proof based on Green's identity.
The results extend to multivariate parabolic equations with first order terms and monotonically increasing univariate 
data. Related convexity criteria are derived from this new perspective.  The univariate data are assumed to have some upper bound on
exponential growth. Extensions of the argument to lognormal coordinates allow for applications beyond power 
options. Representation formulas of Greeks are implied.  Multivariate comparison with univariate data  
is used  in order to determine optimal strategies for multivariate passport options. These optimal strategies reveal new 
features of multivariate passport
options compared to univariate passport options due to correlation effects. Passport option values are determined by HJB-Cauchy 
problems with control spaces of measurable bounded functions.  Especially the values of  optimal strategy functions of  
passport options, where the control space of strategy values forms a hypercube, are located on the  inverse images of the vertices 
of that cube under the rotation matrix mapping  which defines the diagonalization of the correlation matrix of the underlying assets.   
Especially,  multivariate passport options cannot be reduced to
lookback options as  in the univariate case. However multivariate passport options inherit from univariate passport options the
feature that optimal strategies prescribe switching between long and short limit positions on a high frequency basis. This 
corresponds to control spaces of measurable functions, where more than H\"older regularity cannot be expected. As this is 
often not feasible, it is interesting to introduce  a new product of
 symmetric  passport options with positive trading position constraints. The comparison result is applied to this new product 
 in the case of one underlying share (next to a money account), where optimal strategies prescribe a maximal limit position 
 in the lower asset at each time. Hence, coefficients related to optimal  strategies are not continuous in general, but in 
 a more regular class which is more interesting from the trading perspective and  from the point of view of regularity theory than  the case of classical passport options.

\end{abstract}

\section{Introduction}

Natural sufficient conditions for the existence of smooth densities for parabolic equations of second order  were established 
by H\"ormander in \cite{H}. Later in the 1980's Gaussian  priori estimates of these densities and of multivariate  spatial, time and mixed 
derivatives of arbitrary order were established from the point of view of Malliavin calculus in \cite{KS}. Our interest 
in this paper is to combine these results with Green's identity and related properties of the adjoint of fundamental solutions 
in order to obtain comparison results for multivariate pure diffusions with univariate convex data. This extends generalizations of 
Hajek's univariate comparison results obtained in \cite{KC}. The H\"ormander condition seems natural in this context, because
\begin{description}
\item[a)] 
stronger conditions of degeneracy may lead to regularity constraints or even non-existence of densities,
\item[b)]
comparison of jump 
diffusions does not hold in general since it does not hold for simple Poisson processes as is shown in \cite{V}.
\end{description}
Furthermore, the 
restriction of univariate data is essential as the result cannot be extended to multivariate data. Here we give a different proof 
of a result obtained in \cite{T}. We also extend convexity criteria to the class of highly degenerate parabolic equations  of  pure 
second order.  No-go results of comparison for  multivariate data results are strengthened also in the sense that exponential upper 
bounds for the univariate data are allowed. This allows for applications to power options, where the payoff has polynomial growth 
conditions in lognormal coordinates which transfers to exponential growth conditions for normal coordinates. Even these growth 
conditions can be weakened, as we show in section 4. Fortunately, comparison results with univariate data are sufficient in order 
to determine optimal strategies for  multivariate forms of classical passport options, and of a new class of symmetric passport 
options which are considered in the essentially univariate case of one share and one money account. A review of the literature on 
passport options and an introduction to this new product is given in the last section of this paper. Additional applications 
concerning the representations of Greeks are considered in section 5.

\section{Comparison for univariate data and the adjoint of the fundamental solution}

First, let us consider pure diffusions of second order on the domain $D=[0,T]\times {\mathbb R}^n$, 
where $T>0$ is arbitrarily large and and ${\mathbb R}^n$ is the $n$--dimensional Euclidean space, i.e., we consider Cauchy problems of the form
\begin{equation}\label{cp00}
\left\lbrace \begin{array}{ll}
Lv\equiv v_{t}-\sum_{i,j=1}^na_{ij}\frac{\partial^2 v}{\partial x_i\partial x_j}=0\\
\\
v(0,x)=f(x),
\end{array}\right.
\end{equation}
on the domain $D$ with spatially dependent  coefficients $a_{ij}:{\mathbb R}^n\rightarrow {\mathbb R}$ 
which form a nonnegative coefficient matrix $(a_{ij})\geq 0$ at each $x\in {\mathbb R}^n$.  
Since we have local convexity violation for nontrivial  diffusions with multivariate payoffs (see below), 
and since some optimal control problems of practical interest can be formulated with restricted forms of payoffs, 
we are interested especially in univariate data, i.e., data of the form
\begin{equation}\label{h1}
f(x)=h(x_1)
\end{equation}
after appropriate renumeration of the components of $x=(x_1,\cdots,x_n)^T$. Note that convexity criteria and comparison transfer to stochastic sums.
Uniform ellipticity simplifies the argument a bit, but our methods apply to H\"ormander diffusions as well, and we will 
prove our results in the latter case. We denote the fundamental solution of (\ref{cp00}) by $p$.
 The adjoint equation is
\begin{equation}\label{adjoint00}
 \begin{array}{ll}
L^*u\equiv u_{t}+\sum_{i,j=1}\frac{\partial^2}{\partial x_i\partial x_j}(a_{ij}u)=0,
\end{array}
\end{equation}
where we denote the fundamental solution of (\ref{adjoint00}) by $p^*$. For any regular functions $u,v\in C^{1,2}$, we have Green's identity for variable coefficients
\begin{equation}\label{addiff0}
vL^*u-uLv=(uv)_t-\sum_{i=1}^n\frac{\partial}{\partial x_i}\left[\sum_{j=1}^n
\left(u a_{ij}\frac{\partial v}{\partial x_i}-va_{ij}\frac{\partial u}{\partial x_j}-uv\frac{\partial a_{ij}}{\partial x_j}\right)\right].
\end{equation} 
Now let $u,v$ be the fundamental solutions of $Lv=0$ and $L^*u=0$ respectively. We consider
\begin{equation}
\begin{array}{ll}
v(\sigma,z)=p(\sigma,z;s,y),~u(\sigma,z)=p^*(\sigma,z;t,x),~ t>s.
\end{array}
\end{equation}
Let $\epsilon>0$ be small enough such that $s+\epsilon<t-\epsilon$, and let  $B_R=\{z\in {\mathbb R}^n||z|\leq R\}$,
where $R$ is much larger than $|x|,|y|$. Integrating equation (\ref{addiff0}) over $[s+\epsilon,t-\epsilon]\times B_R$ and using $Lv=0$ 
and $L^*u=0$,  the left side of (\ref{addiff0}) becomes zero, and we get the equation
\begin{equation}\label{addiff1}
\begin{array}{ll}
\int_{B_R}u(t-\epsilon,z)v(t-\epsilon,z)-u(s+\epsilon,z)v(s+\epsilon,z)dz\\
\\
=\int_{s+\epsilon}^{t-\epsilon}\int_{\partial B_R}(\sum_{i=1}^n\frac{\partial}{\partial x_i}\left[\sum_{j=1}^n
\left(u a_{ij}\frac{\partial v}{\partial x_i}-va_{ij}\frac{\partial u}{\partial x_j}-uv\frac{\partial a_{ij}}{\partial x_j}\right)\right](\sigma,z)dSd\sigma,
\end{array}
\end{equation} 
where $\partial B_R$ denotes the boundary of $B_R$, and where the right side of (\ref{addiff1}) is a surface integral over the boundary 
$\partial B_R$ of the ball $B_R$. We are interested in conditions, where the right side of (\ref{addiff1}) converges to zero as 
$R\uparrow \infty$. We then have
\begin{equation}\label{addiff2}
\begin{array}{ll}
\int_{{\mathbb R}^n}u(t-\epsilon,z)v(t-\epsilon,z)dz=\int_{{\mathbb R}^n}u(s+\epsilon,z)v(s+\epsilon,z)dz,\\
\\
\mbox{or}~\\
\\
\int_{{\mathbb R}^n}v(t-\epsilon,z)p^*(t-\epsilon,z;t,x)dz=\int_{{\mathbb R}^n}u(s+\epsilon,z)p(s+\epsilon;s,y)dz.
\end{array}
\end{equation} 
In the limit $\epsilon \downarrow 0$ we get
\begin{equation}
v(t,x)=u(s,y),
\end{equation}
and, similarly for any $h\in {\mathbb R}^n$, we have
\begin{equation}
v(t,x+h)=u(s,y+h).
\end{equation}
 Next, for directions $h$, consider finite difference quotients 
\begin{equation}
D^{+}_hu(s,w)=\frac{u(s,w+h)-u(s,w)}{h}, D^{-}_hu(s,w)=\frac{u(s,w)-u(s,w-h)}{h},
\end{equation}  
and
\begin{equation}
 D^2_hu(s,w)=\frac{D^+_hu(s,w)-D^-_hu(s,w)}{h},
\end{equation}
and for multiindices $\alpha$,
let $D^{\alpha}_h$ denote the coordinate versions of these difference equations. Then
\begin{equation}\label{uvalpha}
 D^{\alpha}_hv(t,x)= D^{\alpha}_hu(s,y)
\end{equation}
holds for all $0\leq |\alpha|\leq 2$, and if the finite difference of the right side of  (\ref{addiff1}) goes to zero as $R\uparrow \infty$ 
and $h\downarrow 0$, then we get indeed
\begin{equation}
 D^{\alpha}_xv(t,x)= D^{\alpha}_yu(s,y)
\end{equation}
Such a formula can then be used in partial integration in order to obtain comparison results, but as convex functions usually 
have some growth at spatial infinity, such partial integrations are useful only for approximating functions, i.e., for data which live 
in a suitable functions space such as  $H^2\cap C^2$, where $H^2$ denotes the standard  Sobolev space of order $2$, i.e., the space where 
the derivatives up to second order are in $L^2$.\\

Before we introduce the H\"ormander conditions, we consider a simple set of conditions such that comparison holds. Here, we mention the 
more specific condition of uniform ellipticity, since this is a typical condition used in the literature. Indeed, for some applications, 
such as passport options, global regular existence  results for the associated Hamilton-Jacobi-Bellmann equations are unknown or
may not exist if the control space is not regular (as is the case especially  for classical  multivariate passport options). Next, we state the uniform ellipticity condition.
\begin{itemize}
\begin{item}[(C)] A uniform ellipticity condition holds, i.e., there exist $0<\lambda <\Lambda<\infty$ such that for all $s^l_i,~ 1\leq i\leq n$ and $\sigma \in[0,T]$
\begin{equation}
\lambda|z|^2\leq \sum_{ij}a_{ij}z_iz_j\leq \Lambda |z|^2.
\end{equation}
The second order coefficients $a_{ij}$ themselves and their partial derivatives up to second order are bounded and H\"older continuous. 
\end{item}\\

Next, we consider  the univariate data condition .
\begin{item}[(D)] For a finite constant $c>0$, the convex univariate data $f$ satisfy for some small $\epsilon >0$ and all $x\in {\mathbb R}$
\begin{equation}\label{ub}
|f(x_1)|\leq c\exp\left(c|x_1|^{2-\epsilon} \right) .
\end{equation}
\end{item}
\end{itemize}

\begin{rem}
Note the  constant $c$ in the exponent in (\ref{ub}). For lognormal coordinates $x_1=\ln(S_1)$  a power option payoff  $f(S_1)=S_1^m=\exp(mx_1)$ 
satisfies the condition D.
The argument for a  comparison considered here can be adapted to this situation.
\end{rem}

Next, let us reconsider the assumption C above.
H\"older continuity and boundedness of second derivatives of $a_{ij}$ 
implies that a)  we have Lipschitz continuous matrix entries $\sigma=(\sigma_{ij})$ such that $\sigma\sigma^T=(a_{ij})$, and that b) the adjoint density 
exists and is positive (as the density itself). Lipschitz continuity of $\sigma$ implies that ordinary stochastic ODE-theory is available and implies 
the existence of strong solutions of associated stochastic ODEs without  reference to ellipticity conditions.  Furthermore, the 
comparison arguments below depend essentially on an integrated Green's identity which leads to relation of finite differences of a density and its 
adjoint, and a priori estimates such that the right side in (\ref{addiff1}) goes to zero as the radius $R$ goes to infinity. 
Hence this argument does not 
depend on uniform ellipticity of the operator. Indeed, what is essentially needed  is the existence of a smooth density and 
appropriate a priori upper bounds such that the right side of the Green's identity   in (\ref{addiff1}) goes to zero as the radius $R$ 
goes to infinity. Hence  the H\"ormander estimates which are strengthened a bit by 
Kusuoka and Stroock in the context of Malliavin calculus (cf. \cite{KS}) are the appropriate estimates. We next introduce this 
class of highly degenerate parabolic equations defined by H\"ormander.\\

 For positive  natural numbers $m,n$, consider a matrix-valued function 
\begin{equation}\label{vcoeff}
x\rightarrow (v_{ji})(x),~1\leq j\leq n,~0\leq i\leq m,~x\in {\mathbb R}^n
\end{equation}
 on ${\mathbb R}^n$, and $m$ smooth vector fields 
\begin{equation}\label{vvec}
V_i=\sum_{j=1}^n v_{ji}(x)\frac{\partial}{\partial x_j},
\end{equation}
where $0\leq i\leq m$. These vector fields define a
Cauchy problem on $[0,\infty)\times {\mathbb R}^n$ for the distribution $p$ of the form
\begin{equation}
	\label{hoer1}
	\left\lbrace \begin{array}{ll}
		\frac{\partial p}{\partial t}=\frac{1}{2}\sum_{i=1}^mV_i^2p+V_0p\\
		\\
		p(0,x;y)=\delta_y(x),
	\end{array}\right.
\end{equation}
where $\delta_y(x)=\delta(x-y)$ is the Dirac delta distribution with an argument shifted by the vector $y\in {\mathbb R}^n$. The density with 
arguments $(t,x)$ is smooth on $(0,\infty)\times {\mathbb R}^n$  for all parameters  $s<t$ and $y\in {\mathbb R}^n$ if
for all $x\in {\mathbb R}^n$  we have
\begin{equation}
H_x={\mathbb R}^n.
\end{equation}
Here, for each $x\in {\mathbb R}^n$ the set $H_x=\cup_{n=0}^{\infty}H^n_x$ is defined inductively as follows.
 For $n=0$ let
\begin{equation}\label{Hoergenx}
\begin{array}{ll}
H^0_x:=\mbox{span}{\Big\{} V_i(x)|1\leq i\leq m {\Big \}},
\end{array}
\end{equation}
and given $H^n_x$ for $n\geq 0$ define

\begin{equation}\label{Hoergenx2}
\begin{array}{ll}
H^{n+1}_x:=H^n_x\cup \mbox{span}{\Big\{}  \left[V_j,V_k \right](x),
 |~0\leq j,k \leq m {\Big \}},
\end{array}
\end{equation}
and where $\left[.,.\right]$ are the Lie bracket of vector fields. We say that the H\"ormander condition (H)  for (\ref{hoer1}) is satisfied, if
\begin{equation}\label{HC}
\forall x\in {\mathbb R}^n~H_x:=\cup_{n=0}^{\infty}H^n_x={\mathbb R}^n.
\end{equation}
 Usually this goes with the assumption that
the coefficients of the vector fields are smooth (i.e., $C^{\infty}$) and bounded  with bounded derivatives, i.e., 
$v_{ji} \in C_{b}^{\infty}\left({\mathbb R}^n \right)$. Linear growth for the functions $v_{ji}$ themselves is allowed such that we use 
the weaker assumption of coefficients with linear growth and bounded derivatives of arbitrary order or
$v_{ji} \in C_{b,l}^{\infty}\left({\mathbb R}^n \right)$ in symbols. If the equation in (\ref{hoer1}) equals a pure diffusion without 
drift as in (\ref{cp00}), then we speak of a pure H\"ormander diffusion. Note that in this case
\begin{equation}\label{ph}
\frac{1}{2}\sum_{i=1}^mV_i^2+V_0=\sum_{ij}^na^h_{jk}\frac{\partial^2}{\partial x_j\partial x_k }
\end{equation}
for some matrix-valued function $x\rightarrow (a^h_{jk}(x))$ which is elliptic at each $x\in {\mathbb R}^n$. Here the upper script $h$ just 
reminds us that the coefficient matrix $(a^h_{ij})$ represents the second order coefficient matrix of a pure H\"ormander diffusion.  The relation 
to diffusion processes is via $(a_{jk})=\sigma\sigma^T$, where the condition H ensures that the latter condition exists.
 The main result in \cite{KS} extending the analysis in \cite{H} is

\begin{thm}
	\label{stroock}
	Consider a $d$-dimensional diffusion process of the form
	\begin{equation}\label{stochm}
		\mathrm{d}X_t \ = \ \sum_{i=1}^d\sigma_{0i}(X_t)\mathrm{d}t + \sum_{j=1}^{d}\sigma_{ij}(X_t)\mathrm{d}W^j_t
	\end{equation}
with $X(0)=x\in {\mathbb R}^d$ with values in ${\mathbb R}^d$ and on a time interval $[0,T]$.
Assume that $\sigma_{0i},\sigma_{ij}\in C^{\infty}_{lb}$. Then the law of the process $X$ is absolutely continuous with respect to the 
Lebesgue measure, and the density $p$ exists and is smooth, i.e. 
\begin{equation}
\begin{array}{ll}
p:(0,T]\times {\mathbb R}^d\times{\mathbb R}^d\rightarrow {\mathbb R}\in C^{\infty}\left( (0,T]\times {\mathbb R}^d\times{\mathbb R}^d\right). 
\end{array}
\end{equation}
Moreover, for each nonnegative natural number $j$, and multiindices $\alpha,\beta$ there are increasing functions of time
\begin{equation}\label{constAB}
A_{j,\alpha,\beta}, B_{j,\alpha,\beta}:[0,T]\rightarrow {\mathbb R},
\end{equation}
and functions
\begin{equation}\label{constmn}
n_{j,\alpha,\beta}, 
m_{j,\alpha,\beta}:
{\mathbb N}\times {\mathbb N}^d\times {\mathbb N}^d\rightarrow {\mathbb N},
\end{equation}
such that 
\begin{equation}\label{pxest}
\begin{array}{ll}
{\Bigg |}\frac{\partial^j}{\partial t^j} \frac{\partial^{|\alpha|}}{\partial x^{\alpha}} \frac{\partial^{|\beta|}}{\partial y^{\beta}}p(t,x,y){\Bigg |}\\
\\
\leq \frac{A_{j,\alpha,\beta}(t)(1+x)^{m_{j,\alpha,\beta}}}{t^{n_{j,\alpha,\beta}}}\exp\left(-B_{j,\alpha,\beta}(t)\frac{(x-y)^2}{t}\right).
\end{array}
\end{equation}
Moreover, all functions (\ref{constAB}) and  (\ref{constmn}) depend on the level of iteration of Lie-bracket iteration at which the H\"{o}rmander condition becomes true.
\end{thm}

Theorem \ref{stroock}  is also sometimes formulated in a probabilistic manner. We note
\begin{cor}
	In the situation of Theorem \ref{stroock} above, solution $X_t^x$ starting at $x$ is in the standard Malliavin space $D^{\infty}$, and there are 
	constants $C_{l,q}$ depending on the derivatives of the drift and dispersion coefficients such that for some constant $\gamma_{l,q}$
\begin{equation}\label{xprocessest}
|X_t^x|_{l,q}\leq C_{l,q}(1+|x|)^{\gamma_{l,q}}.
\end{equation}
Here $|.|_{l,q}$ denotes the norm where derivatives up to order $l$ are in $L^q$ (in the Malliavin sense).
\end{cor}

However, the upper bounds  of spatial derivatives obtained in  \cite{KS}  are not finite on the whole space, an effect which limits existence 
of regular global solutions of  H\"ormander difffusions. Since these upper bounds are natural in general, we have no comparison without 
existence of global regular solutions. This has to be investigated on a case by case basis. For our method we need H\"older continuity for spatial 
derivative up to second order for  the solution of the the pure H\"ormander diffusion problem. Therefore we include existence of global 
regular solutions in the following assumption.
\begin{itemize}
\item[(HE)] We assume that the H\"ormander condition in (\ref{HC}) is satisfied and that we have a pure H\"ormander diffusion, i.e., that 
the condition in (\ref{ph}) is satisfied for the coefficient matrix $(a_{ij})$ in (\ref{cp00+++}) below. Furthermore we assume that the Cauchy problem in   
(\ref{cp00+++})  has a global classical solution which is in $C^{2,\alpha}$,  i.e., a classical solution in $C^{1,2}$ with finite spatial H\"older 
norms for spatial derivatives up to second order on the whole domain of ${\mathbb R}^n$.
\end{itemize}
 We remark that there is no loss of generality if we assume the coefficient matrices to be symmetric. The main comparison results is
\begin{thm} Let $(a_{ij}),(a'_{ij})$ be two matrices of component functions which equal  $n\times n$-matrices $\sigma\sigma^T$ and $\sigma'\sigma'^T$ respectively, 
and such that the conditions in Theorem \ref{stroock} hold for $\sigma_{ij}$. assume that $f(x)=h_1(x_1)$ satisfies condition D. Assume that 
condition(HE is satisfied for $(a_{ij}),(a'_{ij})$ . On the domain $[0,T]\times {\mathbb R}^n$ consider a  Cauchy problem of the form
\begin{equation}\label{cp00+++}
\left\lbrace \begin{array}{ll}
Lv'\equiv v_{t}-\sum_{i,j=1}^na_{ij}\frac{\partial^2 v'}{\partial x_i\partial x_j}=0,\\
\\
v'(0,x)=h_1(x_1),
\end{array}\right.
\end{equation}
and an analogous Cauchy problem with coefficient functions $((a'_{ij})$ respectively. If $(a_{ij})\leq (a'_{ij})$ in the sense that $(a'_{ij})-(a_{ij})$ is a 
nonnegative matrix,  i.e.,  has nonnegative eigenvalues, 
and $a_{11}<a'_{11}$, then $v<v'$ on the domain $(0,T]\times {\mathbb R}^n$.
\end{thm}
The existence assumption in (HE) can be eliminated if we know existence for other reasons, e.g., if the assumptions D and C hold, because uniform 
ellipticity and regularity of coefficients together with the growth and data condition in D implies existence. Therefore, we have
\begin{cor} Let $(a_{ij}),(a'_{ij})$ be two matrices of component functions which satisfy the condition C, and assume that $f(x)=h_1(x_1)$ satisfies 
condition (D). We compare the solution of (\ref{cp00}) with data (\ref{h1}) with the solution of the Cauchy problem
\begin{equation}\label{cp00+-}
\left\lbrace \begin{array}{ll}
Lv'\equiv v'_{t}-\sum_{i,j=1}^na_{ij}\frac{\partial^2 v'}{\partial x_i\partial x_j}=0,\\
\\
v'(0,x)=h_1(x_1),
\end{array}\right.
\end{equation}
If $(a_{ij})\leq (a'_{ij})$ in the sense that $(a'_{ij})-(a_{ij})$ is a nonnegative matrix,  i.e.,  has nonnegative eigenvalues, and $a_{11}<a'_{11}$, 
then $v<v'$ on the domain $(0,T]\times {\mathbb R}^n$.
\end{cor}

\begin{proof}(Martingale theorem.)
For a small $\epsilon$, consider a convolution $h_{1,\epsilon}=h_1\ast_{sp}G_{\epsilon}\in C^{\infty}$ and such that $h_{1,\epsilon}$ remains convex 
and $\lim_{\epsilon\downarrow 0}h_{1,\epsilon}=h_1$ (pointwise). For example $G_{\epsilon}$ may be chosen to be a heat kernel, i.e., the fundamental 
solution  of $p_t-\epsilon  p_{x_1x_1}=0$. Next, for small $\delta >0$ and large $R$ define 
\begin{equation}
h_{1,\delta}^R=:\left\lbrace\begin{array}{ll}
h_{1}(x)~\mbox{if}~|x|\leq R\\
\\
h_{1}(x)\exp(-\delta|x-R|^2)~\mbox{if}~|x|> R,
\end{array}\right.
\end{equation}
and let $h^R_{1,\epsilon,\delta}=h^R_{1,\delta}\ast_{sp}G_{\epsilon}$. 
The latter is not convex anymore, but convex on the large interval $[-R,R]$, and we have $\lim_{\epsilon,\delta\downarrow 0,R\uparrow \infty} 
h^R_{1,\epsilon,\delta}=f$ pointwise. Let  $p,p'$ be 
the fundamental solutions of $p_t- \sum_{i,j}^na_{ij}p_{x_ix_j}=0$ and of  $p'_t- \sum_{i,j}^na'_{ij}p'_{x_ix_j}=0$.
The value functions $v,v'$ have the representations
\begin{equation}\label{vv}
v(t,x)=\int_{{\mathbb R}^n}
h_1(y_1)p(t,x;0,y)dy,
\end{equation}
and
\begin{equation}\label{vv1}
v'(t,x)=\int_{{\mathbb R}^n}
h_1(y_1)p'(t,x;0,y)dy,
\end{equation}
and the approximative value functions have the representations
\begin{equation}\label{veps}
v^{\epsilon,\delta,R}(t,x)=\int_{{\mathbb R}^n}
h_{1,\epsilon,\delta}^R(y_1)p
(t,x;0,y)dy
\end{equation}
and
\begin{equation}\label{veps1}
v'^{\epsilon,\delta,R}(t,x)=\int_{{\mathbb R}^n}
h_{1,\epsilon,\delta}^R(y_1)p'
(t,x;0,y)dy.
\end{equation}
We have $p,p'>0$ by assumption $C$ and there exists an adjoint fundamental solution, i.e., $\left\langle Lv,u \right\rangle= 
\left\langle v,L^*u \right\rangle$, where $L$ is the operator of (\ref{cp00+-}), and $L^*$ is the adjoint operator.
In order to use partial integration and the adjoint, we introduce rotated coordinates $\tilde{x}$  such that $\tilde{x}\rightarrow 
h_{1,\epsilon,\delta}(\tilde{x})$ becomes a function on the whole domain of ${\mathbb R}^n$ which is in $H^2\cap C^2$ with respect 
to this multidimensional domain.  Since $x_1\rightarrow h_{1,\epsilon,\delta}^R(x_1)$ is univariate, we can always consider small 
rotations such that $\tilde{x}=\sum_{j=1}^n\lambda_jx_j$, where $\lambda_j>0,~ 1\leq j\leq n$ and such that $\lambda_1$ is close 
to $1$ and $\lambda_j,j\neq 0$ are small. Let $\tilde{p}$ and $\tilde{p}'$ denote the corresponding  fundamental solution in rotated 
coordinates, i.e., corresponding to $p,p'$ via the rotation transformation outlined.
Next, we use the relation (\ref{addiff1}) which leads to (\ref{uvalpha}), or for multiindices $0\leq |\alpha|\leq 2$ we have
\begin{equation}\label{adeq*}
D^{\alpha}_{\tilde{x}}\tilde{p}
(t,\tilde{x};s,\tilde{y})=D^{\alpha}_{\tilde{y}}\tilde{p}^{*}
(s,\tilde{y};t,\tilde{x}),~ t >s,
\end{equation}
where $\tilde{p}^{*}$ is the adjoint of $\tilde{p}$.
Let $\tilde{y}\rightarrow h_{1,\epsilon,\delta}^{R,\lambda}(\tilde{y})=h_{1,\epsilon,\delta}^{R,}(y_1)$ denote the rotational transform of the data. Then we have
\begin{equation}\label{pp***}
\begin{array}{ll}
D^{\alpha}_{\tilde{x}}\tilde{v}^{\epsilon,\delta,R}(t,\tilde{x})=\int_{{\mathbb R}^n}
h_{1,\epsilon,\delta}^{R,\lambda}(\tilde{y})D^{\alpha}_{\tilde{x}}p
(t,\tilde{x};0,y)dy\\
\\=\int_{{\mathbb R}^n}
h_{1,\epsilon,\delta}^{R,\lambda}(\tilde{y})D^{\alpha}_{\tilde{y}}\tilde{p}^{*}
(0,\tilde{y},t,\tilde{x})d\tilde{y}\\
\\
=\int_{{\mathbb R}^n}
\left( D^{\alpha}_{\tilde{y}}h_{1,\epsilon,\delta}^{R,\lambda}(\tilde{y})\right)\tilde{p}^{*}
(0,\tilde{y},t,\tilde{x})d\tilde{y}.
\end{array}
\end{equation}
Note that $\tilde{p}^{*}>0$ by (\ref{adeq*}) and assumption (C).
Now let $\tilde{A}=\left(\tilde{a}_{ij}\right)$ and   $\tilde{A}'=\left(\tilde{a}'_{ij}\right)$ denote the coefficient matrix in transformed 
rotated coordinates such that for all $1\leq i,j\leq n$ and all $\tilde{x},x$, we have $\tilde{a}_{ij}(\tilde{x})=a_{ij}(x)$ and  
$\tilde{a}'_{ij}(\tilde{x})=a_{ij}'(x).$  Assume that   $\tilde{p}$ is fundamental solution of 
\begin{equation}
\tilde{u}^{\epsilon,\delta,R}_{t}-\mbox{Tr}\left(\tilde{A} D^2\tilde{u}^{\epsilon,\delta,R} \right)=0,
\end{equation}
where we consider this problem along with some data $\tilde{u}(0,\tilde{x})=h_{\epsilon,\delta}^{R,\lambda}(\tilde{x})$, and where
\begin{equation}
\mbox{Tr}\left(\tilde{A} D^2\tilde{u}^{\epsilon,\delta,R} \right)=\sum_{j=0}^n\tilde{A}_{ij}\left( D^2_{jk}\tilde{u}^{\epsilon,\delta,R}\right) \delta_{ik},
\end{equation} 
and compare this to a solution
\begin{equation}
\tilde{u}'^{\epsilon,\delta,R}_{t}-\mbox{Tr}\left(\tilde{A}' D^2\tilde{u}'^{\epsilon,\delta,R} \right)=0
\end{equation}
with the same data.
For $\delta \tilde{u}=\tilde{u}'-\tilde{u}$, we get
\begin{equation}\label{deltau}
\delta \tilde{u}^{\epsilon,\delta,R}_{t}-\mbox{Tr}\left(\tilde{A}' D^2\delta \tilde{u}^{\epsilon,\delta,R} \right)=
\mbox{Tr}\left((\tilde{A}'-\tilde{A}) D^2\tilde{u}^{\epsilon,\delta,R} \right).
\end{equation}
As we have  zero data for the difference, we have
\begin{equation}\label{delta2*}
\delta \tilde{u}^{\epsilon,\delta,R}(t,\tilde{x})=\int_0^{t}\int_{{\mathbb R}^{n}}\mbox{Tr}\left((A'-A) 
D^2\tilde{u}^{\epsilon,\delta,R} \right)(s,\tilde{y})\tilde{p}'(t,\tilde{x};s,\tilde{y})d\tilde{y} ds.
\end{equation}
Now let $(t,\tilde{x})$ be given. For $\frac{R}{2}\gg |\tilde{x}|$ and $\delta,\epsilon >0$ small, we observe from (\ref{pp***})
 that for given $T$ we may choose $R$ large enough
\begin{equation}
\frac{\partial^2 \tilde{u}^{\epsilon,\delta,R}}{\partial \tilde{x}_1^2}
 (s,\tilde{y})>0,~ |\tilde{y}|\leq \frac{R}{2},~s\in [0,T]
\end{equation}
where for $\lambda_1$ close to $1$ and $\lambda_i,~ 2\leq i\leq n$ small this term dominates the other Greeks. We conclude that
\begin{equation}
\left(\left( \tilde{a}'_{11}- \tilde{a}_{11}\right)\frac{\partial^2 \tilde{u}^{\epsilon,\delta,R}}{\partial \tilde{x}^2}\right) 
 (s,\tilde{y})>0,~ |\tilde{y}|\leq \frac{R}{2},~s\in [0,T]
\end{equation}
is the dominating term in 
$\mbox{Tr}\left((A'-A) D^2\tilde{u}^{\epsilon,\delta,R} \right)$
in a ball of radius $\frac{R}{2}$. Outside that ball (since $|\tilde{x}|\ll \frac{R}{2}$) 
$\tilde{p}'(\tau,\tilde{x};s,y)$ becomes small such that 
from (\ref{delta2*}) we get
\begin{equation}\label{delta22}
\delta \tilde{u}^{\epsilon,\delta,R}(\tau,\tilde{x})>0,~ R~\mbox{large},~\epsilon,\delta~\mbox{small,}
\end{equation}
which relation holds also in the limit $\delta,\epsilon \downarrow 0$ and $R\uparrow\infty$.
\end{proof}

Next, we give consider variations of the argument above for comparison and derive convexity criteria for pure H\"ormander diffusions as further corollaries.  A further simple conclusion is that for 
pure H\"ormander diffusions with spatially nonconstant coefficients there are multivariate convex data such that convexity and comparison are locally violated.

For smooth convex data $f:{\mathbb R}^n\rightarrow {\mathbb R}$, consider the function $v^f:=v-f$, where $v$ satisfies the equation (\ref{cp00}).
Since $f$ is independent of time, we obviously  have
 \begin{equation}\label{cp001}
\left\lbrace \begin{array}{ll}
 v^f_{t}-\sum_{i,j=1}^na_{ij}\frac{\partial^2 v^f}{\partial x_i\partial x_j}-\sum_{i,j=1}^na_{ij}\frac{\partial^2 f}{\partial x_i\partial x_j}=0\\
\\
v^f(0,x)=0.
\end{array}\right.
\end{equation}
Hence, we have the representation
 \begin{equation}\label{cp001+}
\begin{array}{ll}
 v^f(t,x)=\int_0^t\int_{{\mathbb R}^n}\left(\sum_{i,j=1}^na_{ij}\frac{\partial^2 f}{\partial x_i\partial x_j}\right)(s,y)p(t,x;s,y)dyds,
\end{array}
\end{equation}
where $p$ is the fundamental solution of $p_{t}-\sum_{i,j=1}^na_{ij}\frac{\partial^2 p}{\partial x_i\partial x_j}=0$.

For data 
\begin{equation}
f_{\epsilon,\delta}^R=:\left\lbrace\begin{array}{ll}
f_{\epsilon}(x)~\mbox{if}~|x|=\sqrt{\sum_{i=1}^nx_i^2}\leq R\\
\\
f_{\epsilon}(x)\exp(-\delta|x-R|^2)~\mbox{if}~|x|=\sqrt{\sum_{i=1}^nx_i^2}> R,
\end{array}\right.
\end{equation}
we get the representation
 \begin{equation}\label{cp0011}
\begin{array}{ll}
\frac{\partial^2 v^{f,R}_{\epsilon,\delta}}{\partial x_i\partial x_j}(t,x)=\int_0^t\int_{{\mathbb R}^n}\left(D^{2}_{y_iy_j}
\left(\sum_{i,j=1}^na_{ij}\frac{\partial^2 f^R_{\epsilon,\delta}}{\partial x_i\partial x_j}\right)(s,y)\right)p^*(t,x;s,y)dyds.
\end{array}
\end{equation}
In the case of univariate data this simplifies considerably to

 \begin{equation}\label{cp0012}
\begin{array}{ll}
\frac{\partial^2 v^{h_1,R}_{\epsilon,\delta}}{\partial x_i\partial x_j}(t,x)=\int_0^t\int_{{\mathbb R}^n}\left(D^{2}_{y_iy_j}
\left(a_{11}\frac{d^2 h^R_{1,\epsilon,\delta}}{d x_1^2}\right)(s,y)\right)p^*(t,x;s,y)dyds.
\end{array}
\end{equation}
From (\ref{cp0011}) we get convexity criteria and failure of convexity of the value function for all nontrivial regular bounded 
coefficient matrices and some convex data. From (\ref{cp0012}) we get partial convexity and comparison for regular bounded coefficients  and univariate convex data. 
In order to state the convexity criterion, we need assumptions for more general  data. We assume
\begin{itemize}
\begin{item}[(D')] For a finite constant $c>0$ the convex  data $f:{\mathbb R}^n\rightarrow {\mathbb R }$ satisfy for some small $\epsilon >0$ and all $x\in {\mathbb R}^n$
\begin{equation}
|f(x)|\leq c\exp\left(c|x|^{2-\epsilon} \right) .
\end{equation}
\end{item}
\end{itemize}
We may abbreviate the coefficient matrix by  $A=(a_{ij})$ and the Hessian of the data $f$ by  $D^2f$. Alexandrov's result (cf. \cite{A})
tells 
us that $D^2f$ exists almost everywhere for convex functions, but in order to have a pointwise well-defined Hessian everywhere, 
we may convolute with a heat kernel of small dispersion, i.e., with fundamental solutions of $q^{\epsilon}_t=\epsilon \Delta q^{\epsilon}$ 
in order to have smooth convex approximations at hand.  We denote $f_{\epsilon}=f\ast q^{\epsilon}$, where $'\ast'$ denotes convolution. 
In the following, we say that a regular function is convex if its Hessian is nonnegative. We say that a function is strictly convex ,
if its Hessian is positive at all arguments of the domain ${\mathbb R}^n$. In the following we assume  that the assumptions D' and HE hold. 
Alternatively, we may  assume that D' and C hold of course.  
Now we have 
\begin{cor}\label{crit1}
Assume D' and HE hold for data and coefficients. Then the solution function $v$ of (\ref{cp00}) is convex if
\begin{equation}
\mbox{Tr}\left(AD^2f_{\epsilon}\right)~\mbox{is convex },
\end{equation} 
for $\epsilon>0$ small, or if the Hessian is a nonnegative $\left(D^2_{x_ix_j}\mbox{Tr}\left(AD^2f_{\epsilon}\right)\right)\geq 0$ on 
the whole domain and for any small $\epsilon >0$.
\end{cor}
\begin{proof}
If the assumptions C and D' hold, then the limit $\delta \downarrow 0$ and $R\uparrow \infty$ of (\ref{cp0011}) exists on both 
sides of the equation for arbitrary small $\epsilon >0$. Since $p^*>0$  by assumptions C and D' and  $\mbox{Tr}\left(AD^2f_{\epsilon}\right)$ 
is convex for $\epsilon >0$ small, this limit of (\ref{cp0011})  shows that the value function $v^f_{\epsilon}$ is convex for small $\epsilon>0$. 
Finally, if $v^f$ is convex, then  $v$ is convex. 
\end{proof}
The latter criterion is sufficient and rather strong in the sense that it essentially says that the time derivative of the value functions has 
a nonegative Hessian. Let us assume that $f\in C^2$ for a moment. In order to have an iff-criterion, we need to consider the critical set $C_r$, 
where $C_r:=\{x|D^2f(x)=0\}$. If $x\in C_r$ and $f$ is convex then $D^2_{x_ix_j}f(x)\geq 0$ such that this $x\in C$ is a minimum for any 
partial second order derivative of $f$. It follows that $D^3_{x_ix_jx_k}f(x)=0$ and $D^4_{x_ix_jx_kx_l}f(x)\geq 0$.The latter very simple but 
effective observation about the behavior of regular convex functions at critical points is also made in \cite{T}. We can combine this observations 
with our observations so far. This means that the 
sufficient criterion of if criterion of Theorem \ref{crit1} becomes an 'iff'-criterion if we restrict the sufficient criterion to the critical set $C_r$. We have
\begin{cor}\label{crit1+}
Assume D' and HE hold for data and coefficients. Then the solution function $v$ of (\ref{cp00}) is convex iff 
\begin{equation}
\forall x\in C_r~ \mbox{Tr}\left(AD^2f\right)~\mbox{is convex in}~ U,
\end{equation} 
where $U$ is a local neighborhood of $x$, or if for all $x\in C_r$ the Hessian is a nonnegative $\left(D^2_{x_ix_j}\mbox{Tr}
\left(AD^2f_{\epsilon}\right)(x)\right)\geq 0$ for any small $\epsilon >0$.
\end{cor}
This criterion is sufficient and necessary, and we can use it in order to obtain a partially different proof of comparison for univariate data. 
Note, however, that the main idea of connecting  Green's identity  for variable coefficients with properties of the relations between (derivatives) 
the fundamental solution and (derivatives) of its adjoint.
 We have
\begin{cor}
Assume D' and HE hold for data and coefficients. Then the local convexity criterion of Corollary \ref{crit1+} holds.
\end{cor} 
\begin{proof}
We assume that $h_1\in C^2$ and that the assumptions  D' and HE hold. Then the limit with $\delta,\epsilon \downarrow 0$ and $R\uparrow \infty$ 
of (\ref{cp0012}) holds  and we have the representation
 \begin{equation}\label{cp0014}
\begin{array}{ll}
\frac{\partial^2 v^{h_1}}{\partial x_i\partial x_j}(t,x)=\int_0^t\int_{{\mathbb R}^n}\left(D^{2}_{y_iy_j}\left(a_{11}\frac{d^2 h_{1}}{d x_1^2}\right)
(s,y)\right)p^*(t,x;s,y)dyds.
\end{array}
\end{equation}
Consider the critical set $C_{h_1}:=\{x|D^2h_1(x_1)=0\}$. If $x\in C_{h_1}$ and $h_1\in C^2$ is convex, then $D^2_{x_ix_j}h_1(x_1)\geq 0$ such that 
this $x\in C_{h_1}$ is a minimum for any partial second order derivative of $h_1$. It follows that $D^3_{x_ix_jx_k}h_1(x_1)=0$ and $D^4_{x_ix_jx_kx_l}
h_1(x_1)\geq 0$, where only $D^4_{x_1x_1x_1x_1}h_1(x_1)\geq 0$ may be different from zero for $x\in C_{h_1}$. We get
\begin{equation}
\forall x\in C_{h_1}~\left(D^{2}_{x_ix_j}\left(a_{11}\frac{d^2 h_{1}}{d x_1^2}\right)(s,x)\right)=a_{11}(s,x)\delta_{1i}\delta_{1j}\frac{\partial^2}
{\partial x_i\partial x_j}\frac{d^2 h_{1}}{d x_1^2}\geq 0
\end{equation}
and the local convexity criterion is satisfied. Comparison follows. 
\end{proof}
Essentially, in the latter corollary we have derived a generalization of the locally convexity preserving condition in \cite{T} where a uniform 
ellipticity condition in \cite{T} is generalized to the condition (HE), and the polynomial upper bound in \cite{T} is generalized to an 
exponential upper bound in D'. Recall this condition. Let $L=:\frac{\partial}{\partial t}-L_{sp}$ with $L$ as in (\ref{cp00}). Assume that HE and D'. 
For regular functions $f:{\mathbb R}^n\rightarrow {\mathbb R}$ and arbitrary 'directions' $u\in {\mathbb R}^n$ define as usual
\begin{equation}
D_uf(x)=\lim_{h\downarrow 0}\frac{f(x+hu)-f(x)}{h},~D_{uu}f(x)=\lim_{h\downarrow 0}\frac{D_uf(x+hu)-D_uf(x)}{h}.
\end{equation}
Then the condition in Corollary \ref{crit1+} can be rephrased by saying that under the condition (D') and (HE) the operator $L$ is called locally convexity preserving at
 $x\in {\mathbb R}^n$  if there exists a neighborhood $U=U(x)$ of $x$ in ${\mathbb R}^n$ with respect to the standard topology such that 
\begin{equation}
\forall u\in {\mathbb R}^n~\forall \mbox{convex }f\in C^2(U)~\left(D_{uu}f(x)=0\Rightarrow D_{uu}\left(L_{sp}f\right)(x)\geq 0\right)
\end{equation}
This is a reformulation of the condition in Corollary \ref{crit1+}  is stated in \cite{T} under the  restricted condition C, and where  a more 
restrictive condition than D or D' was imposed on the data, i.e., in \cite{T} it is assumed that the data have a polynomial upper bound.
Under these restricted conditions equivalent conditions for convexity preservation can be obtained (cf. \cite{T}). We mention that these 
criteria can be generalized to the condition (HE), but the data condition of a polynomially upper bound is essentially used.
However this is enough in order to assert that the local convexity condition is violated for multivariate data $f$ in general a fortiori.
\begin{cor}
Assume D' and HE hold for data and coefficients.Then for all nonconstant coefficients there are data such that convexity of the value function is locally violated.
\end{cor}
\begin{proof}
For polynomially bounded data and the condition C the proof in \cite{T} applies. If the conditions HE and D hold, then a solution can be 
represented by limits of solution functions of parabolic problems which satisfy C and have polynomial bounded data, where the convexity is locally violated.
The violation of convexity is then preserved in the limit.  
\end{proof}
\begin{rem}
Note that we consider the essential case of time-homogeneous models in this paper. Convexity criteria are satisfied for purely time dependent 
coefficients for analogous extensions of condition C, of course.
\end{rem}

\section{Applications to finance I: Representations of  Greeks and beyond  power options}

Consider functions $u,v$ with $Lv=0$ and $L^*u=0$ respectively, where for for fixed $x+h,y+h$ 
\begin{equation}
\begin{array}{ll}
v(\sigma,z)=p(\sigma,z;s,y+h),~u(\sigma,z)=p^*(\sigma,z;t,x+h),~ t>s.
\end{array}
\end{equation}
 Integrating equation (\ref{addiff0}) over $[s+\epsilon,t-\epsilon]\times B_R$ and using $Lv=0$ and $L^*u=0$  we observed that the left side of (\ref{addiff0}) 
 becomes zero, and that the right side of the resulting equation
\begin{equation}\label{addiff1g}
\begin{array}{ll}
\int_{B_R}u(t-\epsilon,z)v(t-\epsilon,z)-u(s+\epsilon,z)v(s+\epsilon,z)dz\\
\\
=\int_{s+\epsilon}^{t-\epsilon}\int_{\partial B_R}(\sum_{i=1}^n\frac{\partial}{\partial x_i}\left[\sum_{j=1}^n
\left(u a_{ij}\frac{\partial v}{\partial x_i}-va_{ij}\frac{\partial u}{\partial x_j}-uv\frac{\partial a_{ij}}{\partial x_j}\right)\right](\sigma,z)dSd\sigma,
\end{array}
\end{equation} 
converges to zero as  $R\uparrow \infty$. We then have
\begin{equation}\label{addiff2g}
\begin{array}{ll}
\int_{{\mathbb R}^n}u(t-\epsilon,z)v(t-\epsilon,z)dz=\int_{{\mathbb R}^n}u(s+\epsilon,z)v(s+\epsilon,z)dz,\\
\\
\mbox{which is}~\\
\\
\int_{{\mathbb R}^n}v(t-\epsilon,z)p^*(t-\epsilon,z;t,x+h)dz=\int_{{\mathbb R}^n}u(s+\epsilon,z)p(s+\epsilon,z;s,y+h)dz\\
\\
\mbox{or}\\
\\
\int_{{\mathbb R}^n}p(t-\epsilon,z;s,y)p^*(t-\epsilon,z;t,x+h)dz\\
\\
=\int_{{\mathbb R}^n}p^*(s+\epsilon,z;t,x)p(s+\epsilon,z;s,y+h)dz
\end{array}
\end{equation} 
In the limit $\epsilon \downarrow 0$ the integrand on the left side  contributes only for $z=x+h$ and the integrand on the right side contributes 
only for $z=y+h$ such that indeed
\begin{equation}
v(t,x+h)=u(s,y+h),
\end{equation}
where $h\in {\mathbb R}^n$ was free fixed choice.

 Hence, the relation holds for finite difference quotients
\begin{equation}
D^{+}_hu(s,y)=\frac{u(s,y+h)-u(s,y)}{h}, D^{-}_hu(s,y)=\frac{u(s,y)-u(s,y-h)}{h}.
\end{equation}  
Similarly for higher order finite differences.
Hence for $h\downarrow 0$ and any multiindex $\alpha$ we get indeed
\begin{equation}\label{pp*alpha}
 D^{\alpha}_xv(t,x)= D^{\alpha}_yu(s,y).
\end{equation}
In the notation above with $h=0$ we have
\begin{equation}
\begin{array}{ll}
v(t,x)=p(t,x;s,y),~u(s,y)=p^*(\sigma,z;t,x),~ t>s.
\end{array}
\end{equation}
Such a formula can then be used in partial integration and for the representations of Greeks. For a Cauchy  problem  of a second order linear equation 
with fundamental solution $p$ and initial data $f$ at time $t_0$
we have the representation
\begin{equation}
v^f(t,x)=\int_{{\mathbb R}^n}f(y)p(t,x;t_0,y)dy=\int_{{\mathbb R}^n}f(y)p^*(t_0,y;t,x)dy
\end{equation}
For spatial derivatives we get the representation (\ref{pp*alpha})
\begin{equation}
v^f(t,x)=\int_{{\mathbb R}^n}f(y)D^{\alpha}_xp(t,x;t_0,y)dy=\int_{{\mathbb R}^n}f(y)D^{\alpha}_yp^*(t_0,y;t,x)dy
\end{equation}
If $f\in H^2\cap C^2$, then the derivatives can be shifted to the payoff functions.
Next we consider a application concerning growth conditions of initial data for problems written in lognormal coordinates, where we observe 
that we can improve on the data assumption (D) if some boundary conditions are satisfied.
In finance,  diffusions of second order are usually considered in lognormal coordinates $x_i=\ln(s_i)$ or $s_i=\exp(x_i)$ 
on the domain $D^s=[0,T]\times {\mathbb R}_+^n$, where $T>0$ is arbitrarily large and and where ${\mathbb R}^n_+$  denotes the set of 
positive real numbers. For the transformed value functions we have
\begin{equation}
v^s(t,s):=v(t,x),~\frac{\partial v^s}{\partial s_i}=\frac{\partial v}{\partial x_i}\frac{dx_i}{ds_i}=\frac{\partial v}{\partial x_i}
\frac{1}{s_i},~s_is_j\frac{\partial^2 v^s}{\partial s_i\partial s_j}=\frac{\partial^2 v}{\partial x_i\partial x_j},
\end{equation}
and the Cauchy problem in (\ref{cp00}) becomes a Cauchy problem on the domain $D^s$ of the form
\begin{equation}\label{cp00s}
\left\lbrace \begin{array}{ll}
L^sv\equiv v^s_{t}-\sum_{i,j=1}^na^s_{ij}s_is_j\frac{\partial^2 v^s}{\partial s_i\partial s_j}=0,\\
\\
v^s(0,s)=f^s(s),
\end{array}\right.
\end{equation}
where we assume univariate data
\begin{equation}\label{h1s}
f^s(s):=f(x)=h^s(s_1):=h(x_1),
\end{equation}
and where
\begin{equation}
a^s_{ij}(s):=a_{ij}(x),
\end{equation}
again after appropriate renumeration of the components of  $s=(s_1,\cdots,s_n)^T$ and $x=(x_1,\cdots,x_n)^T$ respectively.  Here we 
mention that option value problems are often given in the form of a final value problem which can be obtained from (\ref{cp00s}) by a 
time transformation $\tau=T-t$. Here, we stick with Cauchy problem formulation for convenience. Note that the constant $c>0$ in the assumption 
D above is arbitrary such that power options,  i.e., options with payoff $h^s(s_1)=s_1 ^m$ for some integer $m$, 
 are subsumed. However,e if 
we have natural zero boundary conditions of the value functions at $s_i=0$ for all $i$ then we can weaken the data condition to 
\begin{itemize}
\begin{item}[($D^s$)] For a finite constant $c>0$ the convex univariate data $f$ satisfy for some small $\epsilon >0$ and all $s_1\in {\mathbb R}_+$
\begin{equation}
|f^s(s_1)|\leq c\exp\left(c|s_1|^{2-\epsilon} \right) .
\end{equation}
\end{item}
\end{itemize}
In order to observe that this generalization is possible note first that the adjoint equation is
\begin{equation}\label{adjoint00s}
 \begin{array}{ll}
L^*_su\equiv u^s_{t}+\sum_{i,j=1}s_is_j\frac{\partial^2}{\partial s_i\partial s_j}(a^s_{ij}u^s)=0.
\end{array}
\end{equation}
We may denote the fundamental solution of (\ref{adjoint00}) by $p^*_s$. Then for  regular functions $u^s,v^s\in C^{1,2}$ on $[0,T]\times 
{\mathbb R}^n_+$ we have Green's identity for variable coefficients
\begin{equation}\label{addiff0s}
\begin{array}{ll}
v^sL^*u^s-u^sLv^s=(u^sv^s)_t\\
\\
-\sum_{i=1}^ns_i\frac{\partial}{\partial s_i}\left[\sum_{j=1}^n
\left(u^s a^s_{ij}s_i\frac{\partial v^s}{\partial s_i}-v^sa^s_{ij}\frac{\partial u^s}{\partial s_j}-u^sv^ss_j\frac{\partial a^s_{ij}}{\partial s_j}\right)\right],
\end{array}
\end{equation} 
and for fundamental solutions  $u^s,v^s$ to $L_sv^s=0$ and $L^*_su^s=0$ we may integrate this identity 
over $s+\epsilon<t-\epsilon$ with $\epsilon >0$ as before, and over  $B^+_R:=\{s\in {\mathbb R}^n_+||s|\leq R\}$ for large  $R$. 

\section{Applications to finance II: Passport options, symmetric passport options, and optimal strategies }

Options on a traded account have been widely studied in the previous literature. In the simplest setup, the client 
is free to trade in two underlying assets subject to specific contractual limits. At the time of the maturity, he
can keep the profits from this trading strategy while his losses are forgiven. Some of the contracts within this family, like
passport options, have been even actively traded. However, the popularity of such contracts has been rather small. 
For univariate passport options   passport options can be subsumed by lookback options. However, this does not hold
for multivariate passport options, as optimal strategies have a much more complicated structure, a consequence, which
 we draw below from the comparison theorem above. Secondly, mathematical optimal strategies, e.g. for a 
passport call written on one share imply high frequency trading between short and long limit positions. For multivariate
passport options this statement has to be  modified according to the correlation structure of the underlying assets, as we
observe below. High frequency trading between short and long limit positions may be related may be also related
to the opinion that passport options are expensive. But up to possible costs for handling the transactions passport options, the 
costs are just replication costs. Extreme short positions may also be unpopular to such an extent that they may be
restricted by law from time to time.  Another reason may be the fact that the previously considered contracts
 treat the two underlying assets asymmetrically. In this traditional setup, the restriction of the trading position is set to only one asset and the residual wealth
is invested in the second asset. For the passport option, the restriction on the position in the first asset is $[-1,1]$, 
meaning that the agent can take any position between long and short. The position in the second asset is given by the
residual wealth.\\

Passport options were introduced in \cite{HLP}.
The authors derived the optimal strategy in the 
geometric Brownian motion model, which is achieved by a short position when the traded account is negative and a long position 
when the traded account is positive. They also found the corresponding option value by solving the corresponding pricing
partial differential equation. Henderson and Hobson (2000) showed that the same strategy remains optimal in the presence 
of stochastic volatility. Shreve and Vecer (2000) considered more general trading
limits on the first asset. The optimality of the solution was proved using the probabilistic arguments based on a 
comparison theorem of Hajek (1985). Vecer (2001) later showed that Asian options are special cases of options on a
traded account when the restriction on the first asset has a specific deterministic form and found a novel pricing partial differential
equation. Delbaen and Yor (2002) showed that the strategy for the passport option remains optimal when the portfolio rebalancing is
restricted to a discrete time. Kampen (2008) considered  multivariate passport options, where the traded account consists of more than one asset,
and observed that  optimal strategies may depend on the sign of the correlations between assets. However a full determination of optimal strategies 
for multivariate passport options was not given in this note. We shall do this below. Then in a second step we consider a basic example of a symmetric passport 
option with just one share and one money account. We shall discuss the advantages of this idea. We also find interesting optimal strategies very different from
optimal strategies of classical passport options which illustrate the practicability of this new product.\\

First let us extend the results known so far for multivariate passport options. First recall the structure of the product itself.
Given a trading account
\begin{equation}
\Pi=\Pi_{\Delta}=\sum_{i=1}^n \Delta_i S_i,~\mbox{where}~dS_i=\sigma_iS_idW_i,~S(0)=x\in {\mathbb R}^n
\end{equation} 
with $n$ lognormal processes $\left(S_i \right)_{1\leq i\leq n}$, where correlations of Brownian motions $W_i$ are encoded in
$(\rho_{ij})_{1\leq i,j\leq n}$, and $q_i\in [-1,1]$ are bounded trading positions, the price of a classical passport option is given by the the solution 
of an optimal control problem
\begin{equation}
\sup_{-1\leq \Delta_i\leq 1,~ 1\leq i\leq n}E^{x,p}\left(f(\Pi_{\Delta}) \right),~f~\mbox{convex, exponentially bounded},
\end{equation}
for the trading positions $\Delta_i\in [-1,1]$, and where $p$ indicates the initial value of the portfolio variable. Actually, we may assume 
that the volatilities are functions of the assets as long as the regularity assumptions
on the coefficients above in (C) are satisfied. As we shall observe below, the comparison result above then implies that an optimal
strategy maximizes  the basket volatility, i.e.,
\begin{equation}
\sup_{-1\leq \Delta_i\leq 1,~ 1\leq i\leq n}\frac{\sqrt{\sum_{i,j=1}^n\rho_{ij}\Delta_i\Delta_j\sigma_i\sigma_jS_iS_j}}{\sum_{i=1}^nS_i}.
\end{equation}
Hence, signs of correlations (and space-time dependence of the signs of correlations) can change an optimal strategy essentially. This indicates 
also that  multivariate mean comparison results are significant extensions of univariate results. However, the mathematically determined optimal 
strategies are strategies which switch between maximal short and long positions with high frequency such that standard passport option are considered to be expensive. 
Let us go deeper into this result and draw some new consequences. Note that the strategy processes $\Delta$ define a  family of value functions
\begin{equation}
 v^{\delta}(t,s,p) := \E[(\Pi_{\Delta}(T) )^+|S(t) = s, \Pi(t)=p].
\end{equation} 
which satisfy a pure diffusion equations. These value functions can be compared for regular volatility matrices, i.e. regular strategies
$\delta=(\delta_1,\cdots,\delta_n)$  especially, according to the comparison result above.
 For simplicity of notation,  we rewrite the volatility matrix of the underlying assets in the form
\begin{equation}
(\sigma\sigma^T)(\delta):=\left(\delta_i\sigma_iS_i\rho_{ij}\delta_j\sigma_jS_j\right).
\end{equation}
Components of this matrix may be denoted by $(\sigma\sigma^T)_{ij}(\delta)$. Let $\Sigma\Sigma^T(\delta)$ the volatility
matrix where $\sigma\sigma^T(\delta)$ is augmented by the basket volatility term on the diagonal (corresponding to the quadratic variation of $\Pi$) and by the correlation term related to  the correlations of the portfolio variable and the underlyings.
For time to expiration $\tau=T-t$ the passport option value function $v^p$  satisfies the HJB-equation
\begin{equation}\label{hjbclp}
\begin{array}{ll}
\frac{\partial v^p}{\partial \tau}-\sup_{-1\leq \delta_i\leq1, 1\leq i\leq n}Tr\left(\Sigma\Sigma^T(\delta)D^2v^p\right)=0,
\end{array}
\end{equation}
which has to be  solved along with the initial condition $v^p(0,s,p)=(p-K)^+$. Here in the expression $ -1\leq \delta_i\leq1$, $\delta_i$
refers to the value of a strategy function also denoted by $\delta_i$ for simplicity of notation.
Here, $ D^2v^p$ is the Hessian with respect to the variables $(p,s)$.
The supremum of the volatility matrix over the set of regular strategies is outside the regular control space  required (say $C^3_b$) by  
the comparison result in general, but comparison over regular control spaces leads to the monotonicity condition
\begin{equation}\label{movp}
\delta,\delta'\in C^3_b,~(\sigma\sigma^T)(\delta)<(\sigma\sigma^T)(\delta')\Rightarrow v^{\delta}<v^{\delta'}~\mbox{on}~ (0,T]\times{\mathbb R}^n_+
\end{equation}
Here, note that we have expiry at $\tau=0$ which corresponds to $t=T$. The limit or optimal strategy is  a function which is just measurable in general, 
and the order of matrices means that the difference $(\sigma\sigma^T)(\delta')-(\sigma\sigma^T)(\delta)$ is positive definite. \\

Next let us go deeper into the question of existence of global solutions to the HJB-equations. First we have to remark that we cannot expect the 
existence of classical solutions, i.e., solutions which exist in $C^{1,2}$. Standard estimates for classical solutions of second order HJB-Cauchy 
problems even require data in $C^3$ (cf. \cite{Kryl1} and \cite{Kryl2}) in contrast to the Lipschitz-continuous data usually given in finance theory. 
However,  as we have only measurable coefficients in the case of classical passport options, there are deeper reasons that we can only expect 
continuous or H\"older continuous solutions.  Therefore it seems natural to consider such problems in the context of viscosity solutions. 
We do not need to reconsider this theory here, and refer to the clasical reference in \cite{Cr} and to \cite{FS} for the special case of HJB- equations.\\

The difficulties of regularity of solution related to lower regularity of coefficients may be a reason for considering discrete time control 
spaces (as Delbaen did). In this case the semi-group property of the operators ensures that we can reduce the HJB-problem to  Cauchy problems 
with regular coefficients at each  time step where the control remains constant. Superficially, this seems flexible enough as we may have variable
time step sizes and arbitrarily small time step sizes. We may consider a partition of the time horizon interval $[0,T]$.  However, continuous
time control spaces have an interest in their own since we are interested in limit behavior.
 We are interested in limit behavior because it 
reveals new features which are not apparent in discrete models. This is true in many areas of scientific modelling.  
The control space has to be just large enough in order to construct the optimal strategy as a limit from strategies  of this control space,
where the limit has not to be a part of that control space if we know that it is well-defined  for other reasons. \\

For each 
natural number $N\geq 1$ consider a discretization of a finite time horizon $[0,T]$, i.e., a set of adjacent intervals  
$\Delta^S_{T,N}:=\left\{\Delta^T_i|0\leq i\leq 2^ N-1\right\}$, 
where $\Delta^T_i=\left[T\frac{i}{2^N},T\frac{i+1}{2^{N+1}}\right]$.  The comparison result then implies that an optimal  
strategy function $\delta=(\delta_1,\cdots,\delta_n) :\Delta^S_N\times {\mathbb R}_+^n\rightarrow [-1,1]^n$ for a classical 
multivariate passport option on $n$ uncorrelated assets  with a discrete time set $\left\{t_i=\frac{i}{2^N}|0\leq i\leq 2^ N-1\right\}$ 
has its values at the vertices of the hypercube $[-1,1]^n$. Indeed such a strategy is independent from the asset values (a further difference
to symmetric passport options introduced below, which is an interesting feature of the latter ne type of product). For the former product we
may define the strategy function as a piecewise constant step function on the time intervals $\Delta^T_i$ which is Lebesgues measure for each $N$ sure. 
Such a construction lead to natural limit control spaces of measurable functions as $N\uparrow \infty$. Such control spaces may be considered for 
classical passport options.  We call them natural Lebesgues-measurable control spaces. The strategy function members of such 
natural Lebesgue-measurable control spaces can also be constructed from spaces of smooth functions with bounded derivatives, 
and we call the limit HJB- Cauchy problems  with natural Lebesgue-measurable  control spaces the associated HJB-Cauchy problems.\\

Next, what can we expect about solutions of HJB-Cauchy equations with measurable coefficients (related to optimal strategies)? Well, the best estimates 
for parabolic equations in this case are based on the works of Nash and de Giorgi. Especially de Giorgi's estimates were adapted and 
extended to parabolic equations and quasilinear parabolic equations in \cite{LS}. However, these estimates require an uniform
ellipticity condition as in C, and they provide no more than H\"older continuous solutions. We cannot expect more in our
framework of highly degenerate parabolic equations with measurable coefficients. In this generality a regularity and existence 
theory is not available, but viscosity solutions can be established in specific cases. \\

Therefore for the purpose of this paper, it is natural to  impose the following extended assumption for HJB-equations related to multivariate classical passport options.
\begin{itemize}
\item[HEHJB)] For a HJB Cauchy problems in (\ref{hjbclp}) we assume that the condition HE holds for smooth strategy functions $\delta_i$  with bounded 
derivatives. For a natural Lebesgues-measurable Limit  control space of strategy functions we consider  the associated HJB-Cauchy problem. Then this HJB-Cauchy 
problem is assumed to have a unique viscosity solution on the time interval $[0,T]$.
\end{itemize}

Now the matrix order in (\ref{movp}) corresponds to the basket volatility order 
\begin{equation}
\frac{\sqrt{\sum_{i,j=1}^n(\sigma\sigma^T)_{ij}(\delta})}{\sum_{i=1}^nS_i}<\frac{\sqrt{\sum_{i,j=1}^n(\sigma\sigma^T)_{ij}(\delta'})}{\sum_{i=1}^nS_i}
\end{equation}
such that we can indeed consider the order of basket volatilities in order to determine optimal strategies. This is worth noting in our context since the 
stochastic sum $\Pi=\sum_iS^{\Delta}_i:=\Delta_iS_i$ with increment $d\Pi=d(\sum_{i=1}^ndS^{\Delta}_i)=\sum_{i=1}^ndS^{\Delta}_i$ have the representation 
(with some standard Brownian motion $\tilde{W}$)
\begin{equation}\label{basket}
\frac{d\Pi}{\Pi}=\sqrt{\sum_{ij}(\sigma\sigma^T)_{ij}(\Delta,S)}d\tilde{W}.
\end{equation}
Note that  the HJB-equation in (\ref{hjbclp})depends on $n+1$ variables corresponding to the processes $\Pi,S_1,\cdots,S_n$ with a univariate payoff which depends only on the variable $p$ corresponding to the portfolio process $\Pi$. The comparison result above can be applied then directly. We do not need the representation in (\ref{basket}) in order to obtain optimal strategies, but this representation shows that optimal startegies maximize the volatility of the corresponding stochastic sum process, which is remarkable.
With this preparations we get the following result.
\begin{thm}
For a passport call written on $n$ assets as above assume that the assumption HEHJB is satisfied. Then any optimal strategy maximizes the basket volatility function
\begin{equation}
\sigma_B(\delta,s):=\sqrt{\sum_{ij}^n(\sigma\sigma^T)_{ij}(\delta)}=\sqrt{<Qs_{\delta},\Lambda Q s_{\delta}>}
\end{equation}
where  $s_{\delta}=(\delta_1s_1,\cdots,\delta_ns_n)^T$ are the weighted asset values and
 $Q^T\Lambda Q=(\sigma_i\rho_{ij}\sigma_j)$  with eigenvalue matrix $\Lambda=\mbox{diag}(\lambda_i,~ 1\leq i\leq n)$, i.e. $Q=(q_{ij})$ 
 denotes the diagonalization matrix for the volatility matrix  $(\sigma_i\rho_{ij}\sigma_j)$. As a consequence  any  optimal strategy satisfies at any time 
\begin{equation}
\delta^{opt}\in \{Q^T\delta^q_{vc}|\delta^q_{vc}\in V_C\}
\end{equation}
where $V_C$ denotes the set of strategy functions with values in the set of vertices of  the cube $[-1,1]^n$, i.e.,
\begin{equation}
V_C:=\{\delta |~\delta:[0,T]\times{\mathbb R}_+^n\times {\mathbb R}\rightarrow \{-1,1\}^n~\mbox{is a measurable function}\}.
\end{equation}
Here, $\delta^q_{vc}(t,s,p)=(\delta^q_{vc1}(t,s,p),\cdots,\delta^q_{vcn}(t,s,p))$ and for a matrix $Q$ and a vector function $\delta$ the expression $Q\delta=\left(\sum_{j}q_{1j}\delta_j,\cdots,\sum_{j}q_{nj}\delta_j\right)^T$ is understood pointwise.
Optimal strategies are then determined as follows. Start with $(t,s,p)\rightarrow \delta^{opt}_{vc}(t,s,p)=(\delta^{opt}_{vc1}(t,s,p),\cdots,\delta^{opt}_{vcn}(t,s,p))\in V_C$ where each component is determined by the corresponding one-dimensional marginal problem (for each $1\leq i\leq n$ passport option written on $S_i$ where all $S_j$ for $j\neq i$ are set to zero). Then $Q^T \delta^{opt}_{vc}$  is an optimal strategy of the multivariate passport option.
\end{thm}

\begin{proof}
Follows from the assumption HEHJB, the construction of a natural control space of measurable functions and the comparison theorem.
\end{proof}

The result above has the interesting consequence that correlations can have the effect that in  situations of special correlations an 
investor who follows an optimal strategy may switch between long and short limit positions of one asset or some assets and does not 
invest in some or all other assets at all. Nevertheless the high frequency switching strategy between long and short positions 
which is known for marginal univariate problems is somehow preserved in an optimal strategy  of the multivariate problem. From a financial 
point of view multivariate passport options cannot be subsumed by lookback options as in the univariate case, because the optimal strategies
are much more complex. Nevertheless the high frequency shifting of huge amount of sums between large long and short positions survives essentially
at least for some of the underlyings in an optimal strategy. This may be considered as impractical.\\

Furthermore, the fact that the trading constraints on the  assets are asymmetric for classical options on a traded
account limits the applicability of such contracts. If we consider a foreign exchange type option with the underlying currencies 
dollar and euro, the contract with asymmetric constraints would be different from the perspective of the dollar and the euro investor. This is
not the case for the plain vanilla options, where a call option from the perspective of the first currency is 
a put option from the perspective of the second currency. This leads to a natural question, whether we can formulate
an option on a traded account with the symmetrical treatment of the two underlying assets.\\

The approach that treats both assets symmetrically is rather straightforward. Instead of imposing an absolute restriction on the
position in the first asset, one can simply require a relative restriction in terms of the fraction of the current wealth. 
The most natural restriction is to allow the client to invest any proportion of his wealth to the first asset, so the $\alpha$ fraction 
of the invested wealth in that asset is in the interval $[0,1]$. Obviously, this is  symmetric with respect to the second
asset as the residual wealth proportion $1-\alpha$ invested in the second asset is restricted to the same interval $[0,1]$. 
Moreover, this approach generalizes to any number of assets, so we can formulate the symmetric problem for an arbitrary number
of assets $N$. This is a very natural approach as the investors are typically free to invest any portion of their wealth to assets of 
their choice, corresponding to $[0,1]$ fraction of their total wealth. The problem of finding the optimal strategy that maximizes
the option value is rather complex for any $N >2$, and thus we limit ourselves only to $N=2$ in this application section.\\

Imposing symmetric trading restrictions is only the first necessary step for the symmetric treatment of the underlying assets. We
also need to use the reference asset that treats the individual assets symmetrically. A reference asset candidate here is an index
consisting of 50\% of both assets. In the following text, we mathematically formalize the 
definition of the option on a traded account that treats both assets symmetrically. This is model independent. Next, we assume
geometric Brownian motion dynamics and first derive the evolution of the asset prices with respect to the index. In the following
step, we find the evolution of the actively traded account with respect to the index. In order to find the optimal strategy,
we need the above  generalization of Hajek's comparison theorem (which extends   a comparison result for stochastic sums
published in Kampen (2016)) and adapt it to our problem. It is still true that the optimal strategy
has the largest volatility with respect to the index, which is interesting in itself as the resulting portfolio has the largest
price variance and thus it determines the maximal possible distributional departure in the sense of $L^2$ norm from the index that
can be achieved by an active trading. This strategy is well known stop-loss strategy, which invests all the wealth in the weaker
asset.\\

So let us consider  just two underlying  assets, let us call them $S$ and $M$. 
For instance, they can represent a stock market and a money market for the stock market type
contracts, or two currencies in the foreign exchange type contracts.
These are the names of the assets with no numerical value rather than the prices. The price 
$S_M(t)$ of the asset $S$ with respect to the asset $M$ is defined as how many units of
$M$ are needed at time $t$ to acquire a single unit of the asset $S$.
As the price is 
a number representing a relationship of two assets, we systematically use the above two 
asset notation in the following text to reflect it. The asset appearing in the subscript is traditionally
referred to as a reference asset or a numeraire. We will 
 use different reference assets in 
our analysis. For simplicity, we consider only assets that 
have its own martingale measure, such as the stocks that reinvests dividends or the money markets. 
For instance, the prices expressed with respect to a reference asset $M$ are $P^M$ martingales,
in particular $S_M(t)$ is a $P^M$ martingale. Assets that do not have its own martingale measure,
such as the currencies, can be linked to the corresponding money markets using the proper discounting.\\

We can further simplify the setup and introduce the following scaling
$$
S_M(0) = 1.
$$
The investor creates a self-financing portfolio $X$ by starting at $X(0) = S(0) = M(0)$ 
and at time $t$:
\begin{equation}\label{portfolio}
X(t) = \Delta^S(t) S(t) + \Delta^M(t) M(t).
\end{equation}
A natural restriction we consider in this paper is
\begin{equation}\label{constraint}
\Delta^S(t) \geq 0, \qquad  \Delta^M(t) \geq 0,
\end{equation}
so he is not allowed to be short in any of the funds.
The constraint in Equation (\ref{constraint}) means that the investor
is free to invest any fraction between $[0,1]$ of his wealth $X$ into the stock market $S$
with the remaining fraction of his wealth going to the money market $M$. This follows from
 Equation (\ref{portfolio}) by using $X$ as a numeraire:
\begin{equation}
 1 = \Delta^S(t) S_X(t) + \Delta^M(t) M_X(t).
\end{equation}

The lower bound condition in one of the markets imposes an upper bound condition in the 
second market, so the positions are constrained by
\begin{equation}
X_S(t) \geq \Delta^S(t) \geq 0, \qquad  X_M(t) \geq \Delta^M(t) \geq 0,
\end{equation}
This is a very
natural condition. Moreover, it treats both assets equivalently, imposing the same 
restriction. Note that the upper bounds are random and depend on the current value of
the investor's wealth $X$. A trivial observation is that $X$ must be always non-negative with
zero wealth being an absorbing boundary.

\begin{rem}[Relationship to passport options] The constraint for the univaraite  classical passport option is on the position in the stock market only
$$
a \leq \Delta^S(t) \leq b,
$$
the position in the second market $M$ follows from
$$
\Delta^M(t) = X_M(t) - \Delta^S(t) S_M(t).
$$
In particular, it can be negative even in the situation when we constrain the investor to
have a positive position in $S(t)$ by requiring $\Delta^S(t) \geq a \geq 0$. The condition 
is not symmetric for both assets, the imposed restriction does not treat them equivalently.
This is arguably one of the main reasons why such contract is not appealing to the investors. 
Moreover, the traded account can become negative in contrast to the situation that treats
both assets symmetrically.
\end{rem}

For preservation of the symmetry of the contract, it is necessary that the reference asset also treats
both assets equally.  One obvious choice is to use
\begin{equation}\label{index}
N(t) = \tfrac12 (S(t) + M(t)).
\end{equation}
The asset $N$ can be regarded as an index consisting of the two assets, or equivalently, a basket of the two
assets.\\

The contact on the actively traded account can be then defined by a payoff at the terminal time $T$
\begin{equation}
 (X_N(T) - K)^+ \text{ units of } N(T)
\end{equation}
for some contractually defined strike $K$. As $X_N(0) = 1$, the strike that corresponds to the
at the money option is equal to $K=1$. In order to preserve the symmetry, the contract has 
to be settled in the index $N$ rather than a single asset $S$ or $N$. For instance, if the 
contract is written on two currencies, say dollar and euro, the contract seen from the position of the
investor or the euro investor is identical. Next we exemplify the previous discussion in the
conext of a GBM model. We note that the following considerations can be generalized rather straightforwardly to the case of variable volatilities.
 The seller of the option must be ready to cover any trading strategy used by the holder of the 
contract. In order to indicate the dependence of the portfolio on the strategy we may sometimes write $X^{\Delta}:=X$ in the following. The fair price of the contract corresponds to the trading strategy 
$\Delta^S(t)$ that maximizes the expectation of the 
\begin{equation}
 \E^N(X^{\Delta}_N(T) - 1)^+.
\end{equation}
 Let us assume geometric Brownian motion model for the 
stock price $S_M(t)$, so
\begin{equation}
 dS_M(t) = \sigma S_M(t) dW^M(t).
\end{equation}
Any discounting is already incorporated in the money market $M$ and the price $S_M(t)$ is 
$P^M$ martingale. Similarly,
the inverse price
\begin{equation}
 dM_S(t) = \sigma M_S(t) dW^S(t)
\end{equation}
is a $P^S$ martingale. The relationship between $W^M(t)$ and $W^S(t)$ is
\begin{equation}
 dW^S(t) = - dW^M(t) + \sigma dt.
\end{equation}
From the self-financing trading assumption, the evolution of the trading portfolio $X$ is
\begin{equation}
 dX_M(t) = \Delta^S(t) dS_M(t)
\end{equation}
and
\begin{equation}
 dX_S(t) = \Delta^M(t) dM_S(t).
\end{equation}
In order to find the optimal strategy, we need to find price evolutions with respect to the index $N$.
\begin{lem}
 The evolution of the price $M_N(t)$ under the probability measure $P^N$ is given by
 \begin{equation}
 dM_N(t) = \frac12 \sigma M_N(t) (2 - M_N(t))  dW^N(t).
\end{equation}
 \end{lem}
\begin {proof}
Note that 
\begin{equation}
 M_N(t) = \frac{M(t)}{\tfrac12(M(t) + S(t))} = \frac{2}{1 +S_M(t)}
\end{equation}
and thus
\begin{eqnarray*}
 dM_N(t) &=& d \left(\frac{2}{1 +S_M(t)} \right) \\
 &=& - \left(\frac{2}{(1 +S_M(t))^2}\right) \sigma S_M(t) dW^M(t) +  \left(\frac{2}{(1 +S_M(t))^3}\right) \sigma^2 S_M^2(t) dt \notag \\
 &=& \left(\frac{2}{(1 +S_M(t))^2}\right) \sigma S_M(t) \left[-dW^M(t) + \frac{\sigma S_M(t)}{(1 +S_M(t))}dt \right] \notag \\
 &=& \left(\frac{2}{(1 +S_M(t))^2}\right) \sigma S_M(t) dW^N(t) \notag \\
 &=& \sigma M_N(t) \left(\frac{S_M(t)}{(1 +S_M(t))}\right)  dW^N(t) \notag \\
 &=& \frac12 \sigma M_N(t) S_N(t)  dW^N(t) \notag \\
 &=& \frac12 \sigma M_N(t) (2 - M_N(t))  dW^N(t). \notag
 \end{eqnarray*}
The process $M_N(t)$ must be $P^N$ martingale, which determines $W^N(t)$ as
\begin{eqnarray}
 dW^N(t) &=& -dW^M(t) + \frac{\sigma S_M(t)}{(1 +S_M(t))}dt \\
 &=& -dW^M(t) + \frac12 \sigma S_N(t) dt. \notag
 \end{eqnarray}
\end{proof}
 
The SDE in Equation (\ref{index}) is interesting on its own as it represents the evolution of the asset with
respect to the index.  From the definition of $N$ in Equation (\ref{index}), we have
\begin{equation}\label{mnsn}
 2 = M_N(t) + S_N(t),
\end{equation}
constraining the $M_N(t)$ process between 0 and 2:
$$0 \leq M_N(t) \leq 2.$$
One can think about $M_N(t)$ as the scaled proportion of the money market $M$ in the index $N$. The price $M_N(t)$ has
the largest volatility when $M_N(t) = 1$, or in other words, when $M(t) = S(t)$. The process $M_N(t)$ loses volatility
in two extreme cases, when $M_N(t) = 0$ and when $M_N(t) = 2$. The first case corresponds to $S_M(t) = \infty$, so
the asset $M$ is worthless in comparison with the asset $S$, the second case corresponds to $S_M(t) = 0$ when the
asset $S$ is worthless in comparison with the asset $M$.\\

Note that from the symmetry of the problem, we have immediately
\begin{equation}\label{sn}
dS_N(t) = - \frac12 \sigma S_N(t) (2 - S_N(t))  dW^N(t).
\end{equation}
It also follows from Equation (\ref{mnsn}).\\

Now we are ready to compute the evolution of $X_N(t)$. 
\begin{lem}
 The evolution of the actively traded portfolio $X$ with respect to the index $N$ follows:
 \begin{equation}\label{xn}
dX_N(t) = \frac12 \left(X_N(t) - 2 \Delta^S(t)\right) \sigma  S_N(t) dW^N(t).
\end{equation}
\end{lem}
\begin{proof}
We have
\begin{eqnarray*}
 dX_N(t) &=& d(X_M(t) \cdot M_N(t)) \\ &=& X_M(t) dM_N(t) + M_N(t) dX_M(t) + dX_M(t) dM_N(t) \notag \\
 &=& X_M(t) \frac12 \sigma M_N(t) S_N(t)  dW^N(t) + M_N(t) \Delta^S(t) \sigma S_M(t) dW^M(t) \notag \\ && -
  \Delta^S(t) \sigma S_M(t) \frac12 \sigma M_N(t) S_N(t) dt  \notag \\
  &=& \frac12 \sigma X_N(t) S_N(t)  dW^N(t) - \Delta^S(t) \sigma S_N(t) \left[-dW^M(t) + \frac12 \sigma S_N(t)dt\right] 
  \notag \\
  &=& \frac12 \left(X_N(t) - 2 \Delta^S(t)\right) \sigma  S_N(t) dW^N(t). \notag
\end{eqnarray*}
\end{proof}

From $\Delta^S(t) = X_S(t) - \Delta^M(t) M_S(t)$, we also have an alternative representation
\begin{equation}\label{xn2}
dX_N(t) = - \frac12 \left(X_N(t) - 2 \Delta^M(t)\right) \sigma  M_N(t) dW^N(t).
\end{equation}

 For a given convex payoff function $f$ The related symmetric passport option price function $v^{sp}$ has the representaion
\begin{equation}
 v^{\delta}(t,x,y) =\sup_{0\leq\Delta,~0\leq \Delta S\leq X_N}E^{(t,s,x)}\left(f(X_N)\right) .
\end{equation}
 where $S_N(t) = x, X^{\Delta}_N(t) =y$ are initial values of the respective processes at time $t$.

For a given stochastic strategy $\Delta$ we may  define a value function
\begin{equation}
 v^{\delta}(t,x,y) := \E^N[(X^{\Delta}_N(T) - 1)^+|S_N(t) = x, X^{\Delta}_N(t) =y].
\end{equation} 
Let us consider the transformation to normal coordinates  $u^{\delta}(\tau,z_1,z_2):= v^{\delta}(t,z_1,z_2)$, where $\tau=T-t$ $z_1=\ln(x)$ and $z_2=\ln(y)$.
The stochastic strategy $\Delta$ corresponds to a strategy $\delta$ in value space which is a function of the underlyings.
 The function $u^{\delta}$ satisfies the initial- boundary value problem
 \begin{multline}\label{ueq}
u^{\delta}_{\tau} - \frac{1}{8}\sigma^2(2-\exp(z_1))^2u^{\delta}_{z_1z_1}+ \frac{1}{4}\sigma^2(2-\exp(z_1))(\exp(z_1)-2\delta )u^{\delta}_{z_1z_2} \\
 - \frac{1}{8}\sigma^2(\exp(z_1)-2\delta)^2 u^{\delta}_{z_2z_2} = 0.
\end{multline}
with initial condition
\begin{eqnarray}
 u(0,z_1,z_2) = (\exp(z_2)-1)^+).
\end{eqnarray}
We impose natural boundary conditions at spatial  infinity, and  have an additional finite boundary condition at $z_1=\log(2)$.
We get
 \begin{equation}
u^{\delta}_{\tau} - \frac{1}{8}\sigma^2(2-2\delta)^2 u^{\delta}_{z_2z_2} = 0,~\mbox{at}~z_1=\log(2).
\end{equation}
This equation corresponds to the process
\begin{equation} 
dX_N(t)= \frac{1}{2} \left(X_N(t) - 2 \Delta^S(t)\right) \sigma  S_N(t) dW^N(t)
\end{equation}
such that we can apply Hajek's result at the boundary where  $z_1=\log(2)$. Hence we know $\delta =0$  at  $\{(\tau,z_1,z_2)|z_1=\log(2)\}$ a priori.
We may  say that $\delta$ lives in reduced control space if  
$\delta\in C_c:=\{\delta\in C^3|\delta|_{z_1=\ln(2)}=0\}$.
The boundary condition reduces to
 \begin{equation}\label{ueq**}
u^{\delta}_{\tau} - \frac{1}{2}\sigma^2 u^{\delta}_{z_2z_2} = 0,~\mbox{at}~z_1=\log(2),
\end{equation}
and such a boundary condition can be considered if the volatilities are regular functions.
In case of constant volatilities the latter condition simplifies to
\begin{eqnarray}
 u^{\delta}(t,\log(2),z_2) &=& \exp(z_2)\cdot N(d_+) - N(d_-),
\end{eqnarray}
where  $d_{\pm} = \frac{z_2 \pm \frac12 \sigma^2 \tau}{\sigma \sqrt{\tau}}$.
The problem may be considered on the domain $D=[0,T]\times (-\infty,\log(2)]\times {\mathbb R}$.
 There are three further  issues here concerning comparison: a) in which space does the strategy function $\delta$ live?; b) the problem has 
 a boundary in finite space, and comparison has to be adapted to this situation, and c) the spatial part of the operator 
 is not strictly elliptic. We formulate the comparison theorem in regular strategy spaces and for a regularized problem. 
 More precisely, we modify the asset dynamics, where for small $\epsilon >0$ we define
\begin{equation}\label{SN}
dS^{\epsilon}_N=-\frac{1}{2}\sigma S_N(t)\left(2-S_N(t) \right)dW^{N,\epsilon}(t) 
\end{equation}
where $W^{N,\epsilon}(t)$ is constructed by adding a small perpendicular process, i.e.,
\begin{equation}
dW^{N,\epsilon}(t)=dW^N+\epsilon d W^{\perp,N},~\left\langle dW^N,d W^{\perp,N}\right\rangle =0
\end{equation}
The corresponding equation for $u^{\delta,\epsilon}$ gets an additional factor $(1+\epsilon)^2$ in the second term of the
equation (\ref{ueq}) and becomes strictly elliptic. Concerning issue a) we compare $C^3$ strategies 
in order to prove an identity for derivatives of the density   and its adjoint up to second order. The issue in b) is addressed in the proof of the following theorem.
\begin{thm}[Comparison Theorem]
Let $\delta,\delta'\in C^3_c$ and $\epsilon >0$ be strategies of the value functions  $u^{\delta,\epsilon},~u^{\delta',\epsilon}$ defined on the 
domain $D$. Then the order of these value functions is induced by the order of the volatility of the 
portfolio term alone, i.e.,  for $\tau\in (0,T]$ 
\begin{equation}
 \frac{1}{8}\sigma^2(x-2\delta)^2 < \frac{1}{8}\sigma^2(x-2\delta')^2\Rightarrow  u^{\delta,\epsilon}(\tau,.)<u^{\delta',\epsilon}(\tau,.).
\end{equation}
\end{thm}

\begin{proof} 
For small positive angle $\theta$ consider the transformed coordinates
 \begin{equation}
 \left( \begin{array}{ll}
 \tilde{z}_1 \\ \tilde{z_2}
 \end{array}\right) =\left( \begin{array}{ll}
\cos(\theta)& \sin(\theta)\\ -\sin(\theta) & \cos(\theta)
 \end{array}\right) \left( \begin{array}{ll}
 z_1 \\ z_2
 \end{array}\right) 
\end{equation}
Multiplying with the inverse (rotation by $-\theta$),
we observe that $-\sin(-\theta)\tilde{z_1}+\cos(-\theta)\tilde{z_2}=\sin(\theta)\tilde{z_1}+\cos(\theta)\tilde{z_2}=z_2$ such that
both coefficients $\sin(\theta),\cos(\theta)$ in the sum representation of $z_2$ are positive 
for $\theta$ positive and small. 
Let  $u^{\delta,\theta }$ and $u^{\delta',\theta}$ denote the value functions in rotated coordinates. Recall that we have a reduced boundary condition which does not depend on $\delta$ (due to the application of the classical Hajek result on the boundary). 
For the sake of comparison of both functions we can reduce to zero boundary conditions and extend both functions trivially to the whole space. These trivial extensions to the whole space may be denoted still by   $u^{\delta,\theta }$ and $u^{\delta',\theta}$ for simplicity.
The latter reduction  simplifies classical  
representations of the value functions $u^{\delta}$ and $u^{\delta'}$ and classical representations of their derivatives up to second order. Especially, we can avoid the boundary terms in these representations (boundary layer terms).
Here, by classical representations we mean the classical representations of solutions of initial boundary value  problems in terms of the fundamental solutions. 
\begin{rem}
Such reductions to zero boundary conditions seem to be not familiar to all readers in the probabilistic community. We shall give a more detailed 
description of this step in the more interesting case of multivariate symmetric passport options in a subsequent paper.  
\end{rem}

In this context we remark that for almost all regular functions $\delta$ the fundamental solution is well-defined. Hence in these
transformed extended coordinates, the problem is defined on the whole space where initial data are defined as a 
payoff of a weighted sum $ (\exp(z_2)-1)^+= (\exp(\sin(\theta)\tilde{z_1}+\cos(\theta)\tilde{z_2})-1)^+$.  
For a payoff  $f$ 
define for small $\delta_0>0$ and large $R$ an approximation  of the payoff function
\begin{equation}
f_{\delta_0}^R(w)=:\left\lbrace\begin{array}{ll}
f(w)~\mbox{if}~|w|\leq R,\\
\\
f(w)\exp(-\delta_0|w-R|^2)~\mbox{if}~|w|> R,
\end{array}\right. 
\end{equation}
and let $f_{\epsilon,\delta_0}^R$ be a smoothed version of $f_{\delta,0}^R$ (smoothing close to identity). Note that the function  
$f_{\epsilon,\delta_0}^R$ is in $H^2\cap C^2$.
Let $u^{\delta,\theta,\epsilon,\delta_0,R}$ be a value function of the regularized (i.e., strictly 
elliptic approximation) form of the equation (\ref{ueq}) in rotated coordinates with data 
$ f_{\epsilon,\delta_0}^R(\sin(\theta)\tilde{z_1}+\cos(\theta)\tilde{z_2})$,  let
 $p^{\delta,\theta,\epsilon,\delta_0,R}$ be the corresponding 
 fundamental solution, and let  $p^{*,\delta,\theta,\epsilon,\delta_0,R}$ be its adjoint 
 (backward and forward equation density in probabilistic terms). The approximative value function itself 
 and the multivariate spatial derivatives of order $|\alpha|\leq 2$  have essentially the representation 
\begin{equation}\label{veps2}
D^{\alpha}_{\tilde{z}}u^{\delta,\theta,\epsilon,\delta_0,R}(\tau,\tilde{z_1},\tilde{z_2})
=\int_0^{\tau}\int_{{\mathbb R}^n}
f_{\epsilon,\delta}^R(\xi)D^{\alpha}_{\tilde{z}}p^{\delta,\theta,\epsilon,\delta_0,R}
(\tau,\tilde{z_1},\tilde{z_2};\sigma,\xi_1,\xi_2)
d\xi_1 d\xi_2 d\sigma,
\end{equation}
where we can  suppress the identical boundary term due to the reduction indicated above.
Next, for $w=(w_1,w_2)$ and $\sigma<s<\tau$  define 
\begin{equation}
\begin{array}{ll}
\overline{v}(s,w)=p^{\delta,\theta,\epsilon,\delta_0,R}
(s,w;\sigma,\xi_1,\xi_2)\\
\\
\overline{u}(s,w)=p^{*,\delta,\theta,\epsilon,\delta_0,R}
(s,w;\tau,\tilde{z_1},\tilde{z_2}).
\end{array}
\end{equation}
Let $L\overline{v}=0$ and $L^*\overline{u}=0$ abbreviate the equations for $\overline{u},\overline{v}$ (approximative equations for (\ref{ueq})). 
We may assume that $\epsilon >0$ is small enough such that $\sigma+\epsilon <s<\tau-\epsilon$. Integrating over the domain 
$[\sigma+\epsilon ,\tau-\epsilon]\times B_R$ (where $B_R$ is the $2$-dimensional ball of radius $R$ around the origin) we have
\begin{equation}\label{int1}
\begin{array}{ll}
0&=\int_{\sigma+\epsilon}^{\tau-\epsilon}\int_{B_R}(\overline{u}L\overline{v}-\overline{v}L^*\overline{u})(s,w)dwds\\
\\
&=\int_{B_R}(\overline{u}(\tau-\epsilon,w)\overline{v}(\tau-\epsilon,w)-\overline{v}(\sigma+\epsilon,w)\overline{u}(\sigma+\epsilon,w))dw+r^*_{B_R},
\end{array}
\end{equation}
where $\lim_{R\uparrow \infty}r^*_{B_R}=0$ for the reminder term which 
follows from $C^3$-regularity of coefficients
and an a priori estimates of the densities. Next, for directions $h$, consider finite difference quotients 
$D^{+}_h\overline{u}(s,w)=\frac{\overline{u}(s,w+h)-\overline{u}(s,w)}{h}$, $D^{-}_h\overline{u}(s,w)=\frac{\overline{u}(s,w)-\overline{u}(s,w-h)}{h}$,  
 $D^2_h\overline{u}(s,w)=\frac{D^+_h\overline{u}(s,w)-D^-_h\overline{u}(s,w)}{h}$, and for multiindices $\alpha$,
let $D^{\alpha}_h$ denote the coordinate versions of these difference equations. Then from (\ref{int1}), we get for all $0\leq |\alpha|\leq 2$
\begin{equation}\label{int2}
\begin{array}{ll}
0&=\int_{{\mathbb R}^2}((D^{\alpha}_h\overline{u})(\tau-\epsilon,w)\overline{v}(\tau-\epsilon,w)-(D^{\alpha}_h(\overline{v})(\sigma+\epsilon,w)
\overline{u}(\sigma+\epsilon,w))dw\\
\\
&=\int_{{\mathbb R}^2}((D^{\alpha}_h\overline{u})(\tau-\epsilon,w)p^{\delta,\theta,\epsilon,\delta_0,R}
(\tau-\epsilon,w;\sigma,\xi_1,\xi_2)\\
\\
&-(D^{\alpha}_h\overline{v})(\sigma+\epsilon,w)p^{*,\delta,\theta,\epsilon,\delta_0,R}
(\sigma+\epsilon,w;\tau,\tilde{z_1},\tilde{z_2}))dw
\end{array}
\end{equation}
We conclude that 
\begin{equation}
(D^{\alpha}_{\tilde{z}}\overline{v})(\tau,\tilde{z})=(D^{\alpha}_{\xi}\overline{u})(\sigma,\xi).
\end{equation}
Hence we can do partial integration and obtain comparison for arbitrary small $\theta$ which is preserved in the limit $\theta\downarrow 0$.
\end{proof}
%
%
The regular control  space $C^3$ does not contain the volatility-maximizing function  $ I(z_1 \leq 0)= I\exp(z_1)\leq1)$ or  $ I(x \leq 1)$ of the  portfolio term.
Define the sequence of functions $h^{\epsilon}$, where
\begin{equation}
h^{\epsilon}(z_1)=\left\lbrace \begin{array}{ll}
1~\mbox{if $z_1\leq -\epsilon$},\\
\\
\exp\left(-1-\frac{\epsilon}{z_1}\right)~\mbox{if $-\epsilon \leq z_1\leq 0$},\\
\\
0~\mbox{else}.
\end{array}\right.
\end{equation}
Let $\delta^{\epsilon}\in C^{\infty}$ be defined by a convolution of $h^{\epsilon}$ with a smoothing Gaussian kernel which is close to identity. 
Then this a sequence of functions $\delta ^{\epsilon}$  which is monotonically increasing as $\epsilon$ decreases and 
$\lim_{\epsilon\downarrow 0}\delta^{\epsilon}=\delta^{opt}= I(x \leq 1)$. According to the Comparison Theorem we have
\begin{equation}
 v^{\delta^{\epsilon}}(t,x,y) = \lim_{\epsilon\downarrow 0} \E^N[(X^{\Delta^{\epsilon}}_N(T) - 1)^+|S_N(t) = x, X^{\Delta^{\epsilon}}_N(t) =y]
\end{equation}
and stochastic ODE theory shows that the limit  $v^{\delta^{opt}}(t,x,y)$ exists as $\epsilon \downarrow 0$.

\begin{thm}[Optimal strategy] The optimal strategy maximizing $\E^N[X_N(T) - K]^+$is given by
\begin{equation}
 d\bar{X}_N(t)
 = \frac12 \left(S_N(t) - 2  \cdot I(S_N(t) \leq 1))\right) \sigma  \bar{X}_N(t) dW^N(t).
\end{equation}
\end{thm}
\begin{proof}
According to the Comparison theorem, the optimal strategy maximizes the absolute value of the 
$dW^N(t)$ term. The optimal position $\Delta^S(t)$  is attained at one of the ends of 
the interval for its possible range. 
When $\Delta^S(t) = 0$, the absolute value reduces to $$X_N(t).$$ When $\Delta^S(t) = X_S(t)$, the 
absolute value is equal to 
$$
-X_N(t) + 2 X_S(t). 
$$
Thus $\Delta^S(t) = 0$ is optimal when
$$
X_N(t) \geq -X_N(t) + 2X_S(t), 
$$
which is equivalent to
$$
S(t) \geq M(t).
$$
\begin{equation}
 \bar\Delta^S(t) = \begin{cases}
               X_S(t), & S(t) \leq M(t), \\
               0, & S(t) \geq M(t),
               \end{cases}
\end{equation}
and
\begin{equation}
 \bar\Delta^M(t) = \begin{cases}
               0, & S(t) \leq M(t), \\
               X_M(t), & S(t) \geq M(t),
               \end{cases}
\end{equation}
More succinctly, 
\begin{equation}
  \bar\Delta^S(t) = \bar{X}_S(t) \cdot I(S_N(t) \leq 1), \qquad  \bar\Delta^M(t) = \bar{X}_M(t) \cdot I(S_N(t) \geq 1).
\end{equation}
Thus it is optimal to be fully invested in the weaker asset. The evolution of the optimal portfolio is 
given by
\begin{eqnarray}
 d\bar{X}_N(t) &=& \frac12 \left(\bar{X}_N(t) - 2 \Delta^S(t)\right) \sigma  S_N(t) dW^N(t) \\
\end{eqnarray}
\end{proof}

\newpage

\appendix{Appendix A: Value function representation, the adjoint, and derivatives}
\newline

We consider some observations of Section 3 in more detail. For $x,y,h\in {\mathbb R}^n$ fixed we consider
\begin{equation}
\begin{array}{ll}
v(\sigma,z)=p(\sigma,z;s,y+h),~u(\sigma,z)=p^*(\sigma,z;t,x+h),~ t>s.
\end{array}
\end{equation}
We have  $Lv=0$ and $L^*u=0$ , and  integrating equation (\ref{addiff0}) over $[s+\epsilon,t-\epsilon]\times B_R$ we  have
\begin{equation}\label{addiff1g}
\begin{array}{ll}
\int_{B_R}u(t-\epsilon,z)v(t-\epsilon,z)-u(s+\epsilon,z)v(s+\epsilon,z)dz=\\
\\
\int_{s+\epsilon}^{t-\epsilon}\int_{\partial B_R}(\sum_{i=1}^n\frac{\partial}{\partial x_i}\left[\sum_{j=1}^n
\left(u a_{ij}\frac{\partial v}{\partial x_i}-va_{ij}\frac{\partial u}{\partial x_j}-uv\frac{\partial a_{ij}}{\partial x_j}\right)\right](\sigma,z)dSd\sigma.
\end{array}
\end{equation} 
$a_{ij}$ and $\frac{\partial a_{ij}}{\partial x_j}$ are usually assumed to be bounded. However for the purpose of estimating the right side of (\ref{addiff1g}) it is sufficent that $a_{ij}$ is of linear growth and that all the other assumption in \cite{KS} are matched. In theis case we ahe apolynomial grwoth factor time a Gaussian where we have an  exponential decay as the Euclidean distances $|(\sigma,z)-(s,y+h)|$ and  $|(\sigma,z)-(s,x+h)|$ become large. Hence,  the right side of (\ref{addiff1g}) goes to zero as  $R\uparrow \infty$, and we have
\begin{equation}\label{addiff2g}
\begin{array}{ll}
\int_{{\mathbb R}^n}u(t-\epsilon,z)v(t-\epsilon,z)dz=\int_{{\mathbb R}^n}u(s+\epsilon,z)v(s+\epsilon,z)dz,\\
\\
\mbox{which is}~\\
\\
\int_{{\mathbb R}^n}v(t-\epsilon,z)p^*(t-\epsilon,z;t,x+h)dz=\int_{{\mathbb R}^n}u(s+\epsilon,z)p(s+\epsilon,z;s,y+h)dz
\end{array}
\end{equation} 
In the limit $\epsilon \downarrow 0$  in  the integrand on the left side we have  $p^*(t-\epsilon,z;t,x+h)\rightarrow \delta(z-(x+h))$ and in the integrand on the right side we have $p(s+\epsilon,z;t,x+h)\rightarrow \delta(z-(y+h))$  such that indeed
\begin{equation}
v(t,x+h)=u(s,y+h),
\end{equation}
where $h\in {\mathbb R}^n$ was free fixed choice.

 Hence, the relation holds for finite difference quotients
\begin{equation}
D^{+}_hu(s,y)=\frac{u(s,y+h)-u(s,y)}{h}, D^{-}_hu(s,y)=\frac{u(s,y)-u(s,y-h)}{h},
\end{equation}  
and 
\begin{equation}
D^{2}_{h_ih_j}u(s,y)=\frac{D^+_{h_i}u(s,y)-D^-_{h_i}u(s,y)}{h_j}=\frac{u(s,y+h_i)-2u(s,y)+u(s,y-h_i)}{h_jh_i}
\end{equation}
for second order finite differences. Applying  regularity of $v$ and $u$ ,
 for $h\downarrow 0$ (or $h_i,h_j\downarrow 0$ and any multiindex $\alpha$ we get indeed
\begin{equation}\label{pp*alpha}
 D^{\alpha}_xv(t,x)= D^{\alpha}_yu(s,y).
\end{equation}

For the understanding of Greek representations it is useful to reconsider for any fixed $h\in {\mathbb R}^n$ we first observe
\begin{equation}\label{addiff2g*}
\begin{array}{ll}
\lim_{\epsilon\downarrow 0}\int_{{\mathbb R}^n}u(t-\epsilon,z)v(t-\epsilon,z)dz=\lim_{\epsilon\downarrow 0}\int_{{\mathbb R}^n}u(s+\epsilon,z)v(s+\epsilon,z)dz,\\
\\
\mbox{iff}~\\
\\
p(t,x+h;s,y+h)=\lim_{\epsilon\downarrow 0}\int_{{\mathbb R}^n}p(t-\epsilon,z;s,y+h)p^*(t-\epsilon,z;t,x+h)dz\\
\\
=\lim_{\epsilon\downarrow 0}\int_{{\mathbb R}^n}p^*(s+\epsilon,z;t,x+h)p(s+\epsilon,z;s,y+h)dz=p^*(s,y+h;t,x+h)\\
\\
\mbox{for}~ t>s.
\end{array}
\end{equation}
Let $\{e_i\}_{1\leq i\leq n}$ be the Eucledean basis of ${\mathbb R}^n$. For any given $1\leq i\leq n$ and $h_i$ the mean value theorem inplies that for some $y^*_{ih1},y^*_{ih2}\in [y,y+h_ie_i]$ we have 
\begin{equation}\label{pa}
\begin{array}{ll}
\lim_{h_i\downarrow 0}\frac{p(t,x+h_ie_i;s,y+h_ie_i)-p(t,x;s,y+h_ie_i)}{h_i}\\
\\
=\lim_{h_i\downarrow 0}\frac{p(t,x+h_ie_i;s,y))-p(t,x;s,y)+p_{y_i}(t,x+h_ie_i,s,y^*_{ih1})h_i-p_{y_i}(t,x;s,y^*_{ih2})h_i}{h_i}\\
\\
=\lim_{h_i\downarrow 0}\frac{p(t,x+h_ie_i;s,y))-p(t,x;s,y)}{h_i}\\
\\
+\lim_{h_i\downarrow 0}(p_{y_i}(t,x+h_ie_i,s,y^*_{ih1})-p(t,x;s,y^*_{ih2}))\\
\\
=p_{x_i}(t,x;s,y)
\end{array}
\end{equation}
Similarly
\begin{equation}\label{pb}
\begin{array}{ll}
\lim_{h_i\downarrow 0}\frac{p^*(s,y+h_ie_i;t,x+h_ie_i)-p^*(s,y;t,x+h_ie_i)}{h_i}\\
\\
=\lim_{h_i\downarrow 0}\frac{p^*(s,y+h_ie_i;t,x))-p^*(s,y;t,x)+p^*_{x_i}(s,y+h_ie_i,t,x^*_{ih1})h_i-p_{x_i}(s,y;t,x^*_{ih2})h_i}{h_i}\\
\\
=\lim_{h_i\downarrow 0}\frac{p^*(s,y+h_ie_i;t,x))-p^*(s,y;t,x)}{h_i}\\
\\
+\lim_{h_i\downarrow 0}(p^*_{y_i}(s,y+h_ie_i,t,x^*_{ih1})-p(s,y;t,x^*_{ih2}))\\
\\
=p^*_{y_i}(s,y;t,x).
\end{array}
\end{equation}
Invoking (\ref{addiff2g*}) the terms in (\ref{pa}) and (\ref{pb}) are equal, i.e., for all $t>s$, $x,y\in {\mathbb R}^n$
\begin{equation} 
p_{x_i}(t,x;s,y)=p^*_{y_i}(s,y;t,x).
\end{equation}
Then we may start with these expressions and shifted versions
\begin{equation}
p_{x_i}(t,x+h;s,y+h)=p^*_{y_i}(s,y+h;t,x+h),
\end{equation}
and repeat the argument above which leads to identities as in (\ref{pp*alpha}).
The finite difference difference approximations above can be used in order to develop computation schemes for the Greeks (with derivative shifts for $f\in C^{|\alpha|}\cap H^{|\alpha|}$
\begin{equation}
D^{\alpha}_xv^f(t,x)=\int_{{\mathbb R}^n}f(y)D^{\alpha}_xp(t,x;t_0,y)dy=\int_{{\mathbb R}^n}f(y)D^{\alpha}_yp^*(t_0,y;t,x)dy.
\end{equation}


\newpage
\begin{thebibliography}{19}
 \baselineskip=12pt
 {


\bibitem{A}{{\sc Alexandrov, A.D.} {``Almost everywhere existence ofthe second differential of a convex function and some properties of convex 
functions connected with it,''} {\em  Leningrad State Univ. Annals,} Math. Ser. (6), 3--35, 1939.}



\bibitem{Cr}
{\sc  Crandall, M.G., Ishii, H., Lions, P.-L. }, {\em A user's guide to viscosity solutions}, Bull.A.M.S., N.S., p.1-67, 1992.

\bibitem{DY}{{\sc  Delbaen, F. and M. Yor,} {``Passport options,''} {\em Mathematical Finance,} {\bf 12}(4), 299--328, 2002.}



\bibitem{E}
{\sc  Elliott, R.J., Chan, L., Siu, T.K.}, {\em  A Dupire equation for a regime switching model }, Int. J. Theor. Appl. Finance, Vol. 18(4), 2015. 


\bibitem{FS}
{\sc  Fleming, W., Soner, H.M.}, {\em Controlled Markov processes and Viscosity solutions}, 2nd ed., Springer, 2006.


\bibitem{Ha}
{\sc  Hayek, J.}, {\em Mean stochastic comparison results of diffusions}, Z. f. Wahrscheinichkeitstheorie u. verw. Gebiete, vol. 68.  p. 315-329, 1985.



\bibitem{HH}{{\sc Henderson, V. and D. Hobson,} {``Local time, coupling and the passport option,''} {\em Finance and Stochastics,} 
{\bf 4}(1), 69--80, 2000.}




\bibitem{H}
	\textsc{H\"{o}rmander, L.}: {\em Hypoelliptic second order differential equations}, {\em Acta Math.,} Vol. 119,  147-171, 1967.

	
	
\bibitem{HLP}{{\sc
Hyer, T., A. Lipton-Lifschitz, and D. Pugachevsky,} {``Passport to Success: Unveiling a new class of options that offer principal protection to actively managed funds,''}
{\em Risk}, 1997, No. 10, 127--132.}




\bibitem{KPP}
{{\sc Kampen, J.,} {``On Optimal Strategies of Multivariate Passport Options,''} {\em  Springer},  No. 12, 643--649, ISSN 612-3956, 2008.}



\bibitem{KC}{{\sc Kampen, J.,} {``Generalisation of Hajek's stochastic comparison results to stochastic sums,''} {\em  Int. J. of Stochastic Analysis}, Vol. 2016 (6p).}




\bibitem{KS} 
{\sc Kusuoka, S., Stroock, D.}: {\em Application of Malliavin calculus II} J. Fac. Sci. Univ. Tokio, Sect. IA, Math. 32, p. 1-76, 1985.




\bibitem{Kryl1} 
{\sc Krylov, N. V.}: {\em Controlled diffusion processes,} Springer N.Y., 1980.



\bibitem{Kryl2} 
{\sc Krylov, N. V.}: {\em Nonlinear elliptic and parabolic equations of second order,} Reidel Dordrecht, 1987.




\bibitem{LS} 
{\sc Ladyzhenskaja, O.A., Uraltseva, N. N.}: {\em Boundary problems for linear and quasilinear parabolic equations I+ II,} Izvest.
Akad. Nauk. SSSR, 26, p. 5-52, p. 753-780, 1962.





\bibitem{SV}{{\sc
Shreve, S. E., and J. Vecer,} 
 {``Options on a traded account: Vacation calls, vacation puts and passport options,''} {\em Finance and Stochastics,}  {\bf 4}(3), 
 255--274, 2000.}
 
 
 

\bibitem{T}
{\sc Tusk} {\em Preservation of convexity criteria of solutions to parabolic equations}, J. of Diff. Eq., 206 , 182-226, 2004. 




\bibitem{Vecer1}{{\sc Vecer, J.,} { ``A new PDE approach for pricing arithmetic average Asian options,''}
{\em Journal of Computational Finance,}  {\bf 4}(4), 105--113, 2001.}




\bibitem{V}{
{\sc Vecer, J., Xu, M.} {`` Mean stochastic comparison Theorem cannot be extended to Poisson case,''} J. of Appl. Probability, 41 (4), 1199-1202, 2004. 
}




\bibitem{W}
{\sc Wiener, N.}, {\em Differential space}, Journal of Mathematical Physics,  p. 131-174, 1923.}



\end{thebibliography}
\end{document}